\documentclass[10pt]{article}
\usepackage{amsmath,amssymb,amsthm}
\usepackage{color}
\usepackage{booktabs}
\usepackage{graphicx}
\usepackage{epstopdf}
\usepackage{algorithm,algorithmic}
\usepackage{verbatim}
\usepackage{url}
\usepackage{float}
\graphicspath{{./eps/}}
 \usepackage{array}
 \usepackage{multirow}
\newcolumntype{P}[1]{>{\raggedright\arraybackslash}p{#1}}

\theoremstyle{plain}
\newtheorem{theorem}{Theorem}[section]
\newtheorem{remark}{Remark}[section]

\newtheorem{lemma}{Lemma}[section]
\newtheorem{assumption}{Assumption}[section]
\newtheorem{proposition}{Proposition}[section]

\numberwithin{equation}{section}

\newcommand{\R}{\mathbb{R}}

\newcommand{\E}{\mathbb{E}}

\title{On the Discrepancy Principle for Stochastic Gradient Descent}
\author{ Tim Jahn\thanks{Institute for Mathematics, Goethe-University Frankfurt, Germany (\texttt{jahn@math.uni-frankfurt.de})}
  \and Bangti Jin\thanks{Department of Computer Science, University College London, Gower Street, London WC1E 6BT, UK (\texttt{b.jin@ucl.ac.uk})
  The work of Bangti Jin is supported by UK EPSRC grant EP/T000864/1.}}

\begin{document}

\maketitle

\begin{abstract}
Stochastic gradient descent (SGD) is a promising numerical method for solving large-scale inverse problems. However, its theoretical
properties remain largely underexplored in the lens of classical regularization theory. In this note, we study the classical
discrepancy principle, one of the most popular \textit{a posteriori} choice rules, as the stopping criterion for SGD, and prove the finite-iteration termination property and the convergence of the iterate in probability as the noise level tends to zero. The
theoretical results are complemented with extensive numerical experiments.\\
\textbf{Key words}: stochastic gradient descent, discrepancy principle, convergence
\end{abstract}

\section{Introduction}
In this work, we study the following finite-dimensional linear inverse problem:
\begin{equation}\label{eqn:lininv}
Ax=y^\dagger,
\end{equation}
where $x\in\mathbb{R}^m$ is the unknown signal of interest, $y^\dagger \in\mathbb{R}^n$ is the exact
data and $A\in\mathbb{R}^{n\times m}$ is the system matrix. In practice, we have access only to a
corrupted version $y^\delta$ of the exact data $y^\dag=Ax^\dag$ (with the reference solution $x^\dag$
being any exact solution)
\begin{equation*}
  y^\delta=y^\dagger+\xi
\end{equation*}
where $\xi\in \mathbb{R}^n$ denotes the noise, with a noise level $\delta= \|\xi \|$. In the literature,
a large number of numerical methods have been proposed for solving linear inverse problems accurately
and efficiently (see, e.g., \cite{EnglHankeNeubauer:1996,KaltenbacherNeubauerScherzer:2008,ItoJin:2015}).

When the size of problem \eqref{eqn:lininv} is massive, one attractive method is a simple
stochastic gradient descent (SGD) \cite{RobbinsMonro:1951,BottouCurtisNocedal:2018}. In its simplest
form, it reads as follows: given an initial guess $x_1^\delta = x_1\in\mathbb{R}^m$, let
\begin{equation}\label{eqn:sgd}
  x_{k+1}^\delta: = x_k^\delta - \eta_k((a_{i_k},x_k^\delta)-y^\delta_{i_k})a_{i_k},\quad k=1,2,\ldots,
\end{equation}
where $\eta_k>0$ is a decreasing stepsize, $a_i$ is the $i$-th row of the matrix $A$ (as a column vector),
$(\cdot,\cdot)$ denotes Euclidean inner product on $\mathbb{R}^m$, and the row index $i_k$ at the $k$th
SGD iteration is chosen uniformly (with replacement) from the set $\{1,...,n\}$. It can be derived
by applying stochastic gradient descent to the quadratic functional:
\begin{equation*}
  J(x) = \frac{1}{2n}\|Ax-y^\delta\|^2 = \frac{1}{n}\sum_{i=1}^n f_i(x),\quad \mbox{with } f_i(x) = \frac12((a_i,x)-y_i^\delta)^2.
\end{equation*}
Distinctly, the method \eqref{eqn:sgd} operates only on one single data pair $(a_{i_k},y_{i_k})$ each time, and thus it is directly scalable
to the data size $n$ of problem \eqref{eqn:lininv}. This feature makes it especially attractive in the context of massive data.
In fact, SGD and its variants (e.g., minibatch and accelerated) have been established as the workhorse behind many
challenging training tasks in deep learning \cite{Bottou:2010,BottouCurtisNocedal:2018}, and they are also
popular for image reconstruction in computed tomography \cite{GordonBenderHerman:1970,Natterer:1986}.

Despite the apparent simplicity of the method, the mathematical theory in the lens of classical
regularization theory is far from complete. In the work \cite{JinLu:2019}, the regularizing property
of SGD was proved for a polynomially decaying stepsize schedule, when the stopping index $k$ is determined
\textit{a priori} in relation with the noise level $\delta$. Further, a convergence rate in the mean squared
norm between the iterate $x_k^\delta$ and the exact solution $x^\dag$ was derived, under suitable source
type condition on the ground truth $x^\dag$. These results were recently extended to mildly nonlinear
inverse problems, further assisted with suitable nonlinearity conditions  of the forward
map \cite{JinZhouZou:2020}. However, in these works, the convergence rate can only be achieved under a knowledge of the smoothness
parameter of $x^\dag$, which is usually not directly accessible in practice. Therefore, it is of enormous
practical importance and theoretical interest to develop \textit{a posteriori} stopping rules that do not require such a knowledge.

For deterministic iterative methods \cite{KaltenbacherNeubauerScherzer:2008}, e.g., Landweber method and
Gauss-Newton method, one popular  \textit{a posteriori} stopping rule is
the discrepancy principle, due to Morozov \cite{Morozov:1966}. Specifically, with $x_k^\delta$
being the $k$th iterate constructed by an iterative regularization method, the principle determines the stopping
index $k(\delta)$ by
\begin{equation}\label{eqn:dp}
k(\delta):=\min\left\{ k\in \mathbb{N}: \|Ax_k^\delta-y^\delta\| \le \tau \delta\right\},
\end{equation}
where the constant $\tau>1$ is fixed. {Note that the stopping index $k(\delta)$ depends on the
random iterate $x_k^\delta$, and thus it is also a random variable, which poses the main challenge in the theoretical
analysis.} The use of the discrepancy principle to many deterministic iterative methods
is well understood (see the monograph \cite{KaltenbacherNeubauerScherzer:2008} and the references therein),
but in the context of stochastic iterative methods, it has not been explored so far, to the best of our
knowledge. The goal of this work is to study the basic properties of the discrepancy
principle for SGD. It is worth noting that a direct computation of the residual $\|Ax_k^\delta-y^\delta\|$
at every SGD iteration is demanding. However, one may compute it not at every SGD iteration but only
with a given frequency (e.g., per epoch, see Section \ref{sec:numer}), as done by the popular stochastic variance reduced gradient
\cite{JohnsonZhang:2013}, for which residual evaluation is a part of gradient computation. Also there
are efficient methods to compute the residual $\|Ax_k^\delta-y^\delta\|$ using randomized SVD \cite{KluthJin:2019},
by exploiting the intrinsic low-rank nature for many practical inverse problems.

Now we specify the algorithmic parameters for SGD, and state the main results of the work. Throughout,
we make the following assumption on the stepsizes and the regularity condition on the ground truth solution $x^\dag$,
i.e., the minimum-norm solution defined by
\begin{equation}\label{eqn:sol-min-norm}
  x^\dag=\arg\min_{x:Ax=y^\dag}\|x\|.
\end{equation}
The stepsize schedule in (i) is commonly known as the polynomially decaying stepsize schedule, and (ii) is
the classical power type source condition, where $B=n^{-1}(A^TA)$ {(with $n$ being the data size, i.e., the number of
rows in $A$)}, imposing a type of smoothness
on the solution $x^\dag$ (relative to the system matrix $A$ and the initial guess $x_1$). In the analysis and
computation below, $x_1$ is fixed at $0$. {Generally, in classical regularization theory for infinite-dimensional
inverse problems, the source element $w$ plays the role of a Lagrangian multiplier of the constrained problem in
\eqref{eqn:sol-min-norm}, whose existence is not ensured for an operator
with a nonclosed range and has to be assumed \cite{EnglHankeNeubauer:1996,ItoJin:2015}. In the
finite-dimensional case, the existence of a source element $w$ for the case $p\le \frac12$ is ensured, but the
norm of the source element $w$ can be arbitrarily large.}
\begin{assumption}\label{ass:sgd}
The following conditions hold.
  \item[$\rm(i)$] The stepsizes $\eta_j$ satisfy  $\eta_j=c_0 j^{-\alpha}$, with $\alpha\in(0,1)$ and $c_0\le(\max_{j=1,...,n}\|a_j\|^2)^{-1}$.
  \item[$\rm(ii)$] There is a $p>0$ and a $w\in\mathbb{R}^m$ such that $x^\dagger-x_1=B^pw $.
\end{assumption}

The first theorem gives a finite-iteration termination property of the discrepancy principle, where $\mathbb{P}$
is with respect to the filtration generated by the random index $(i_k)_{k=1}^\infty$. It can also be viewed as
a partial result on the optimality. It implies in particular that for $p<\frac12$, the data propagation error is
of optimal order. The proof relies crucially on the observation that the variance component of the
mean squared residual contributes only marginally for sufficiently large $k$.
\begin{theorem}\label{thm:dp}
Let Assumption \ref{ass:sgd} be fulfilled, and $k(\delta)$ be determined by the discrepancy principle \eqref{eqn:dp}.
Then for all $0<r<1$ and $\tau>\tau^*>1$, with $c=\big(\frac{\tau^*-1}{\sqrt{n}c_p}\big)^{-\frac{2}{(1-\alpha)(\min(2p,r)+1)}}+2$, there holds
\begin{equation*}
 \mathbb{P}\left( k(\delta)\le c \delta^{-\frac{2}{(1-\alpha)(\min\left(2p,r\right)+1)}}\right)\to 1 \quad\mbox{as }\delta \to 0^+,
\end{equation*}
with the constant $c_p=(\frac{(p+\frac12)(1-\alpha)}{c_0e(2^{1-\alpha}-1)})^{p+\frac{1}{2}}\|w\|$.
\end{theorem}

The second contribution of this work is on the convergence in probability of the SGD iterate $x_{k(\delta)}^\delta$
with the stopping index $k(\delta)$ determined by \eqref{eqn:dp}. This result has one drawback. In the proof, we have to assume that the stopping index
$k(\delta)$ is independent of the iterates $x_{k(\delta)}^\delta$. In practice, this can be achieved by running SGD twice with the same data $(y^\delta,
\delta)$: the first round is for the determination of $k(\delta)$, then the second (independent) round is stopped
using $k(\delta)$. This increases the computational expense by a factor of $2$. However, the numerical results in Section
\ref{sec:numer} show that one can use the iterate from the first run without compromising the accuracy.
\begin{theorem}\label{thm:dp:conv}
Let Assumption \ref{ass:sgd} be fulfilled, and $k(\delta)$ be determined by the discrepancy principle \eqref{eqn:dp}. Then for all $\varepsilon>0$  there holds
\begin{equation*}
 \mathbb{P}\left( \| x_{k(\delta)}^\delta-x^\dagger\| \ge \varepsilon\right)\to 0 \quad\mbox{as }\delta \to 0^+,
\end{equation*}
where $(x_k^\delta)_{k\in\mathbb{N}}$ are SGD iterates independent of $k(\delta)$, with the same data $(y^\delta,\delta)$.
\end{theorem}

In sum, Theorems \ref{thm:dp} and \ref{thm:dp:conv} confirm that the discrepancy principle is a valid
\textit{a posteriori} stopping rule for SGD. However, they do not give a rate of convergence, which remains
an open problem. Numerically, we observe that the convergence rate obtained by the discrepancy
principle is nearly order-optimal for low-regularity solutions, as the \textit{a priori} rule
in the regime in \cite{JinLu:2019}, and the performance is competitive with the standard Landweber
method. Thus, the method is especially attractive for finding a low-accuracy solution. However, for very
smooth solutions (i.e., large $p$), it manifested as an undesirable saturation phenomenon, due to the presence of the significant
variance component (when compared with the approximation error), under the setting of Assumption
\ref{ass:sgd}. The rest of the paper is organized as follows. In Sections \ref{sec:bound} and \ref{sec:dp:conv},
we prove Theorems \ref{thm:dp} and \ref{thm:dp:conv}, respectively. Several auxiliary results needed for the
proof of Theorem \ref{thm:dp} are given in Section \ref{sec:auxil}. Finally, several numerical experiments are
presented in Section \ref{sec:numer} to complement the theoretical analysis. We conclude with some useful notation.
We denote the SGD iterate for exact data $y^\dag$ by $x_k$, and that for noisy data $y^\delta$ by $x_k^\delta$.
The expectation $\E[\cdot]$ is with respect to the filtration $\mathcal{F}_k$, generated by the random indices
$\{i_1,\ldots,i_{k}\}$.

\section{The proof of Theorem \ref{thm:dp}}\label{sec:bound}
In this section, we give the proof of Theorem \ref{thm:dp}. First, we give several preliminary facts.
By the construction in \eqref{eqn:sgd}, since $x_k^\delta$ is measurable with respect to $\mathcal{F}_{k-1}$,
\begin{align*}
  \E[x_{k+1}^\delta|\mathcal{F}_{k-1}] &= x_k^\delta-\eta_kn^{-1}\sum_{i=1}^n ((a_i,x_k^\delta)-y_i^\delta)a_i \\
    &= x_k^\delta - \eta_k n^{-1}(A^tAx^\delta_k-A^ty^\delta).
\end{align*}
Thus, by the law of total expectation, the sequence $(\E[x_{k}^\delta])_{k\in\mathbb{N}}$ satisfies the following recursion:
\begin{align}\label{eqn:sgd=lm}
  \E[x_{k+1}^\delta] = \E[x_k^\delta] - \eta_k (\bar A^t\bar A\E[x_k^\delta]-\bar A^t\bar y^\delta)
\end{align}
with $\bar A = n^{-\frac{1}{2}}A$ and $\bar y^\delta=n^{-\frac12}y^\delta$. This is exactly the classical Landweber
method \cite{Landweber:1951} (but with diminishing stepsizes) applied to the rescaled linear system $\bar Ax = \bar y^\delta$. For the Landweber
method, the discrepancy principle \eqref{eqn:dp}, e.g., regularizing property and optimal convergence rates, has been
thoroughly studied for both linear and nonlinear inverse problems (see, e.g., \cite[Chapter 6]{EnglHankeNeubauer:1996}
and \cite{KaltenbacherNeubauerScherzer:2008}). The key insight for the analysis below is the following empirical observation:
for a suitably large $k$, typically the variance component $\E[\|A(x_k^{\delta} - \E [x_{k}^\delta])\|^2]\ll \delta^2$,
as confirmed by the numerical experiments in Section \ref{ssec:variance}. This fact allows us to transfer the results for
the Landweber method to SGD.

The proof of Theorem \ref{thm:dp} employs two preliminary results, whose lengthy proofs are deferred to
Section \ref{sec:auxil}. The first result gives an upper bound of the following stopping index $k^*(\delta)$,
for any $\tau^*>1$, defined by
\begin{equation}\label{eqn:dpex}
k^*(\delta):=\min\{k\in\mathbb{N}~:~ \| A \E [x_{k}^\delta]-y^\delta \|\le \tau^* \delta\}.
\end{equation}
Clearly, $k^*(\delta)$ is the stopping index by the classical discrepancy principle, when applied to
the sequence $(\E [x_{k}^\delta])_{k\in\mathbb{N}}$, which is exactly the Landweber method, in view of the relation \eqref{eqn:sgd=lm}.
\begin{proposition}\label{prop:dp-bdd}
Let Assumption \ref{ass:sgd} be fulfilled. Then for $k^*(\delta)$ defined in \eqref{eqn:dpex}, there holds
\begin{equation}
 k^*(\delta)\le \Big(\frac{\tau^*-1}{\sqrt{n}c_p}  \delta\Big)^{-\frac{2}{(1-\alpha)(2p+1)}}+2,
\end{equation}
with $c_p=(\frac{(p+\frac12)(1-\alpha)}{c_0e(2^{1-\alpha}-1)})^{p+\frac{1}{2}}\|w\|$.
\end{proposition}

The second result gives an upper bound on the variance component $\E[\|A(x_{\kappa(\delta)}^{\delta}-\E [x_{\kappa(\delta)}^\delta])\|^2]$
of the mean squared residual $\E[\|Ax_k-y^\delta\|^2]$. It indicates that the variance $\E[\|A(x_{k(\delta)}^\delta-\E[x_{k(\delta)}^\delta])\|^2]$ contributes only marginally to
the mean squared residual $\E[\|Ax_{k(\delta)}^\delta-y^\delta\|^2]$, and consequently the squared residual $\|Ax_{k(\delta)}^\delta-y^\delta\|^2$
of individual realizations of SGD may be used instead for determining an appropriate stopping index.

\begin{proposition}\label{prop:sgd-variance}
Under Assumption \ref{ass:sgd} with $\kappa(\delta)\ge  \delta^{-\frac{2}{(1-\alpha)(\min(2p,r)+1)}}$ and $0<r<1$, there holds
\begin{equation*}
\E[\|A(x_{\kappa(\delta)}^{\delta}-\E[x_{\kappa(\delta)}^\delta])\|^2] = o(\delta^2),\quad\mbox{as }\delta\to0^+.
\end{equation*}
\end{proposition}

Now we can present the proof of Theorem \ref{thm:dp}.
\begin{proof}
Set $1<\tau^*<\tau$ and $\bar{k}(\delta)= [c\delta^{-\frac{2}{(1-\alpha)(\min(2p,r)+1)}}]+2$ ($[\cdot]$
denotes taking the integral part of a real number), with $c=\big(\frac{\tau^*-1}{\sqrt{n}c_p}\big)^{-\frac{2}{
(1-\alpha)(\min(2p,r)+1)}}$. By the definition of $k(\delta)$ in \eqref{eqn:dp}, the event $\mathcal{E}=
\{k(\delta)\leq \bar{k}(\delta)\}$ is given by
\begin{equation*}
  \mathcal{E} = \{\exists i\in \{1,\ldots,\bar{k}(\delta)\} \mbox{ such that }\|Ax_{i}^\delta-y^\delta\|\leq \tau\delta\}.
\end{equation*}
Thus, $\mathcal{E}\supset\{\|Ax_{\bar{k}(\delta)}^\delta-y^\delta\| \leq \tau\delta\}$. Consequently,
\begin{align*}
  \mathbb{P}(k(\delta)\leq \bar{k}(\delta)) &\ge\mathbb{P}(\|Ax_{\bar{k}(\delta)}^\delta-y^\delta\|\le \tau \delta)\\
   &\ge \mathbb{P}(\|A(x_{\bar k(\delta)}^\delta-\E[x_{\bar{k}(\delta)}^\delta])\| \le (\tau-\tau^*)\delta,~
  \| A\E[x_{\bar{k}(\delta)}^\delta]-y^\delta\| \le \tau^*\delta).
\end{align*}
By the choice of $\bar{k}(\delta)$, Proposition \ref{prop:dp-bdd} implies
\begin{equation*}
  \|A\E[x_{\bar k(\delta)}^\delta]-y^\delta\|\leq \tau^*\delta.
\end{equation*}
Consequently,
\begin{align*}
  \mathbb{P}(k(\delta)\leq \bar{k}(\delta)) & \geq \mathbb{P}(\| A(x_{\bar{k}(\delta)}^\delta-\E[x_{\bar{k}(\delta)}^\delta])\|\le (\tau-\tau^*)\delta)\\
     &=1-\mathbb{P}(\|A(x_{\bar{k}(\delta)}^\delta-\E[x_{\bar{k}(\delta)}^\delta])\| >(\tau-\tau^*)\delta).
\end{align*}
Meanwhile, by Chebyshev's inequality \cite[p. 233]{Feller:1968}, we have
\begin{equation*}
  \mathbb{P}(\|A(x_{\bar{k}(\delta)}^\delta-\E[x_{\bar{k}(\delta)}^\delta])\| >(\tau-\tau^*)\delta) \leq \frac{\E\|A(x_{\bar{k}(\delta)}^\delta-\E [x_{\bar{k}(\delta)}^\delta])\|^2}{(\tau-\tau^*)^2\delta^2}.
\end{equation*}
Therefore,
\begin{align*}
  \mathbb{P}(k(\delta)\leq \bar{k}(\delta)) &\ge 1 - \frac{\E\|A(x_{\bar{k}(\delta)}^\delta-\E[x_{\bar{k}(\delta)}^\delta])\|^2}{(\tau-\tau^*)^2\delta^2},
\end{align*}
which together with Proposition \ref{prop:sgd-variance} directly implies
\begin{align*}
  \mathbb{P}(k(\delta)\leq \bar{k}(\delta)) \to 1 \quad \mbox{as }\delta\to 0^+.
\end{align*}
This completes the proof of the theorem.
\end{proof}


\begin{remark}
The condition $r<1$ is related to an apparent saturation phenomenon with SGD: for any $p>\frac12$, the SGD iterate
$x_k^\delta$ with \textit{a priori} stopping can only achieve a convergence rate comparable with that for $p=\frac12$
in the setting of Assumption \ref{ass:sgd}, at least for the current analysis
\cite{JinLu:2019}. It remains unclear whether this is an intrinsic drawback of SGD or due to limitations of the proof technique.
\end{remark}

\begin{remark}\label{rm:thm1}
In practice, we prefer computing the residual with a frequency $\omega n \in\mathbb{N}$:
\begin{equation*}
k_\omega(\delta):= \min\{ \omega n k~:~k \in \mathbb{N}~,~\| A x^\delta_{\omega n k} - y^\delta\| \le \tau \delta\}.
\end{equation*}
Since one of the numbers $[c\delta^{-\frac{2}{(1-\alpha)(\min(2p,r)+1)}}]+2,....,[c\delta^{-\frac{2}{(1-\alpha)(\min(2p,r)+1)}}]+\omega n+1$ is of the form $\omega n k$, with $k\in\mathbb{N}$, there holds
\begin{equation*}
 \mathbb{P}\left( k_\omega(\delta)\le c \delta^{-\frac{2}{(1-\alpha)(\min\left(2p,r\right)+1)}}+\omega n+1\right)\to 1 \quad\mbox{as }\delta \to 0^+.
\end{equation*}
That is, the upper bound on the stopping index remains largely valid for a variant of the
discrepancy  principle \eqref{eqn:dp} evaluated with a given frequency.
\end{remark}

{\begin{remark}\label{rm:thm:dp}
The finite-iteration termination property in Theorem \ref{thm:dp} relies heavily on the assumption $\alpha<1$
in the definition of the stepsize schedule. Without this condition, Theorem \ref{thm:dp} {\rm(}and thus also the
convergence in probability{\rm)} generally do not hold.
Indeed, if $\mathrm{rank}(A)\ge 2$, $y^\dagger\neq 0$ and $\alpha>1$, then there holds
\begin{equation}\label{eqn:ass-finite}
  \liminf_{\delta\to 0^+}\mathbb{P}(k(\delta)=\infty)>0.
\end{equation}
To prove this assertion, let $k^*\in \mathbb{N}$ be such that $\eta_k\|A\|^2\le \frac12$ for all $k\ge k^*$.
Since $\mathrm{rank}(A)\ge 2$ and $y^\dagger\neq0$, there exists
an index $j\in\{1,\ldots, n\}$ such that $y^\dagger \notin\mathrm{span}( Aa_j)$. In view of the fact
$Ax_k \chi_{\{i_1=\ldots=i_{k^*-1}=j\}} \in \mathrm{span}( Aa_j)$, for $k\in\{1,...,k^*\}$, there exists an
$\eta>0$ with
\begin{equation*}
 \mathbb{P}\left( \|Ax_k-y^\dagger\|\ge \eta,~\forall k\le k^*\right)\ge \mathbb{P}\left( i_1=...=i_{k^*-1}=j\right)>0.
\end{equation*}
Meanwhile for $k> k^*$, similiar to \eqref{prop:dp-es:es1} below, there holds
\begin{align*}
\|Ax_k-y^\dagger\| &\ge \| Ax_{k-1}-y^\dagger\| - \eta_{k-1}|(Ax_{k-1}-y^\dagger,e_{i_{k-1}})| \|AA^te_{i_{k-1}}\|\\
 &\ge ... \ge \| Ax_{k^*}-y^\dagger\| \prod_{i=k^*}^{k-1}(1-\eta_i\|A\|^2).
\end{align*}
Using the elementary inequalities $1+x\le e^x$ for all $x\in\R$ and
$1+x\ge e^{x-x^2}$ for all $x\in [-\frac12,0]$ and the estimate
\eqref{prop:dp-es:es1} below, we deduce
\begin{align*}
\|A x_k^\delta -y^\delta\| & \ge \| A x_k-y^\dagger\| - \|A(x_k-x_k^\delta)-(y^\dagger-y^\delta)\|\\
  &\ge \|Ax_{k^*}-y^\dagger\| \prod_{i=k^*}^{k-1}(1-\|A\|^2\eta_i) - \delta \prod_{i=1}^{k-1}(1+\|A\|^2\eta_i)\\
&\ge \|Ax_{k^*}-y^\dagger\| \exp\Big(-c_0\|A\|^2\sum_{i=k^*}^{k-1}i^{-\alpha} - c_0^2\|A\|^4\sum_{i=k^*}^{k-1} i^{-2\alpha}\Big) - \delta \exp\Big(c_0\|A\|^2\sum_{i=1}^{k-1}i^{-\alpha}\Big)\\
&\ge c^{\prime}\|A x_{k^*}-y^\dagger\| - c''\delta,
\end{align*}
with
\begin{equation*}
  c':=e^{-c_0\|A\|^2\sum_{i=1}^{\infty}i^{-\alpha} - c_0^2\|A\|^4\sum_{i=1}^{\infty} i^{-2\alpha}}>0 \quad
 \mbox{and}\quad c'':=e^{c_0\|A\|^2\sum_{i=1}^\infty i^{-\alpha}}<\infty.
\end{equation*}
So for small enough $\delta>0$, there holds
\begin{equation*}
\|Ax_k^\delta-y^\delta\|\chi_{\{\|Ax_i-y^\dagger\|\ge \eta,~\forall i\le k^*\}} \ge c'\eta - c''\delta > \tau \delta.
\end{equation*}
Consequently,
\begin{equation*}
\liminf_{\delta>0}\mathbb{P}(k(\delta)=\infty) \ge \mathbb{P}\left( \|Ax_i - y^\dagger\|\ge \eta,~\forall i\le k^*\right)>0.
\end{equation*}
This shows the assertion \eqref{eqn:ass-finite}.
\end{remark}}

\section{The proof of Theorem \ref{thm:dp:conv}}\label{sec:dp:conv}

In this section, we prove Theorem \ref{thm:dp:conv}. It employs the following proposition,
which states that potential early stopping actually does not cause any problem.
\begin{proposition}\label{prop:dp-es}
For all $\varepsilon>0$, there is a sequence $(k_\delta^-)_\delta$ with $k_\delta^-\to \infty$ for $\delta \to 0^+$, such that
\begin{equation*}
  \| x_{k(\delta)}^\delta - x^\dagger \|\chi_{\{k(\delta)\le k_\delta^-\}} \le \varepsilon
\end{equation*}
for $\delta>0$ small enough.
\end{proposition}
\begin{proof}
It suffices to show that for all $K\in\mathbb{N}$
 \begin{equation}\label{prop:dp-es:lim}
  \| x_{k(\delta)}-x^\dagger\|\chi_{\{k(\delta)\le K \}}\to 0 \quad \mbox{as } \delta\to0^+.
 \end{equation}
In order to show this, we need the following two estimates for the iterated noise:
\begin{align}
 \|A(x_k^\delta-x_k)-(y^\delta-y^\dagger)\| &\le \delta \prod_{j=1}^{k-1}(1+\eta_j\| A \|^2),\label{prop:dp-es:es1}\\
 \|x_k^\delta-x_k\| &\le \delta \|A\|\sum_{j=1}^{k-1}\eta_j \prod_{i=1}^{j-1}(1+\eta_i\|A\|^2),\label{prop:dp-es:es2}
 \end{align}
with the conventions $\sum_{j=1}^0=0$ and $\prod_{j=1}^0 = 1$. We prove the estimates \eqref{prop:dp-es:es1} and \eqref{prop:dp-es:es2} by
mathematical induction. Note that $a_i=A^te_i$. For the estimate \eqref{prop:dp-es:es1}, by the triangle inequality and the
defining relation \eqref{eqn:sgd} of SGD iteration,
\begin{align*}
  &\quad\|A(x_{k+1}^\delta-x_{k+1})-(y^\delta-y^\dagger)\| \\
  &\le\|A(x_k^\delta-x_k)-(y^\delta-y^\dagger)\| + \eta_k \|\left( A(x_k^\delta-x_k)-(y^\delta-y^\dagger),e_{i_k}\right)AA^te_{i_k}\|\\
   &\le \| A(x_k^\delta-x_k)-(y^\delta-y^\dagger)\|\left(1+\eta_k \|A\|^2\right),
 \end{align*}
and since $x_1=x_1^\delta$, $\| A(x_1^\delta-x_1)-(y^\delta-y^\dagger)\|=\|y^\delta - y^\dagger\| \le \delta$.
For the estimate \eqref{prop:dp-es:es2}, we have $\| x_1^\delta - x_1\|=0$ and
\begin{equation*}
  \|x_{k+1}^\delta -x_{k+1}\| \le \|x_k^\delta-x_k\| + \eta_k \| A \| \|A(x_k^\delta-x_k)-(y^\delta-y^\dagger)\|,
 \end{equation*}
so the claim follows using the estimate \eqref{prop:dp-es:es1}.
Now, for each fixed $K$, since there are only finitely many different realisations of the first
$K$ SGD iterates, there is a (deterministic) $\eta>0$, which depends on $K$, such that
\begin{equation}\label{prop:dp-es:min}
 \min_{k=1,...,K}\left(\| A x_k-y^\dagger \| - \eta\right)\chi_{\{\|Ax_k-y^\dagger\|>0\}}\ge 0,
\end{equation}
where without loss of generality, we have assumed $y^\dagger\neq 0$. Therefore, using
estimates \eqref{prop:dp-es:es1} and \eqref{prop:dp-es:min},
\begin{align*}
  &\quad\|Ax_k^\delta-y^\delta\|\chi_{\{\|Ax_k-y^\dagger\|>0\}} \\
  &\ge \|Ax_k-y^\dagger\|\chi_{\{\|Ax_k-y^\dagger\|>0\}} -\|A(x_k-x_k^\delta)-(y^\dagger-y^\delta)\|\chi_{\{\|Ax_k-y^\dagger\|>0\}}\\
  &\ge \Big(\eta - \delta \prod_{j=1}^{k-1}(1+\eta_j\|A\|^2)\Big)\chi_{\|Ax_k-y^\dagger\|>0\}}> \tau \delta\chi_{\{\|Ax_k-y^\dagger\|>0\}},
\end{align*}
for any $\delta<\frac{\eta}{\tau+\prod_{j=1}^{K-1} (1+\eta_j\|A\|^2)}$. Then by the definition
of the discrepancy principle in \eqref{eqn:dp}, this implies
\begin{equation*}
   \{ k(\delta)\le K\} \subset \{\|Ax_{k(\delta)}-y^\dagger\| = 0\}
  \end{equation*}
  for $\delta>0$ small enough. Meanwhile, since by construction $x_{k(\delta)} \in \mathcal{R}(A^t)=\mathcal{N}(A)^\perp$,
$\|Ax_{k(\delta)}-y^\dag\|=0$ implies $x_{k(\delta)}=x^\dagger$, the minimum norm solution.
The proof of \eqref{prop:dp-es:lim} is concluded by
\begin{align*}
   \|x_{k(\delta)}^\delta -x^\dagger\|\chi_{\{k(\delta)\le K\}} &= \| x_{k(\delta)}^\delta -x_{k(\delta)}\|\chi_{\{k(\delta)\le K\}}\\
                     &\le \delta \|A\| \sum_{j=1}^{K-1} \eta_j \prod_{i=1}^{j-1}(1+\eta_i \|A\|^2) \to 0
\end{align*}
for $\delta\to0^+$, where we have used estimate \eqref{prop:dp-es:es2}.
This completes the proof of the proposition.
\end{proof}

Now we can state the proof of Theorem \ref{thm:dp:conv}.
\begin{proof}[Proof of Theorem \ref{thm:dp:conv}]
Fix $\varepsilon>0$. Proposition \ref{prop:dp-es} and Theorem \ref{thm:dp} guarantee the existence of
two sequences $(k_\delta^-)_\delta,(k_\delta^+)_\delta$, with $k_\delta^- \le k_\delta^+\le c
\delta^{-\frac{2}{(1-\alpha)(\min\left(2p,r\right)+1)}}$, $k_\delta^-\to \infty$ for $\delta\to0^+$
and
\begin{equation*}
 \|x_{k(\delta)}^\delta-x^\dagger\|\chi_{\{k(\delta)\le k^-_\delta\}} \le \varepsilon\quad\mbox{for }\delta\mbox{ small enough }
\end{equation*}
and
\begin{equation*}
  \mathbb{P}\left(k(\delta)\le k^+_\delta\right)\to 1\quad \mbox{for }\delta \to 0^+.
\end{equation*}
Consequently, for $\delta>0$ small enough, there holds
\begin{align*}
 &\quad \mathbb{P}(\|x_{k(\delta)}^\delta- x^\dagger \| > \varepsilon) \\
 &= \mathbb{P}(\|x_{k(\delta)}^\delta - x^\dagger\| >\varepsilon, k(\delta)\le k^-_\delta) + \mathbb{P}(\|x_{k(\delta)}^\delta -x^\dagger\| >\varepsilon, k(\delta)>k^-_\delta)\\
                                                                &=\mathbb{P}(\|x_{k(\delta)} - x^\dagger\| >\varepsilon, k(\delta)>k^-_\delta)\\
                                                                &=\mathbb{P}(\|x_{k(\delta)}-x^\dagger\| >\varepsilon, k^-_\delta<k(\delta)\le k^+_\delta) + \mathbb{P}(\|x_{k(\delta)}-x^\dagger\| >\varepsilon, k(\delta)>k^+_\delta)\\
                                                                &\le\mathbb{P}(\|x_{k(\delta)}-x^\dagger\| >\varepsilon, k^-_\delta<k(\delta)\le k^+_\delta) + \mathbb{P}(k(\delta)>k^+_\delta).
\end{align*}
In view of Theorem \ref{thm:dp}, it remains to show that
\begin{equation*}
  \mathbb{P}(\|x_{k(\delta)}-x^\dagger\| >\varepsilon , k_\delta^-<k(\delta)\le k^+_\delta)\to 0\quad\mbox{ for }\delta\to 0^+.
\end{equation*}
To this end, let $\Omega_\delta:=\{ k^-_\delta\le k(\delta)\le k^+_\delta\}$ and
we split the error into three parts in a customary way: approximation error, data propagation error
and stochastic error. Specifically, by the triangle inequality, there are constants $c_1$ and $c_2$ such that
\begin{align*}
 \| x_{k(\delta)}^\delta-x^\dagger \|\chi_{\Omega_\delta} &= \sum_{k=k^-_\delta}^{k^+_\delta} \| x_k^\delta - x^\dagger\| \chi_{\{k(\delta)=k\}}\\
                                      &\le \sum_{k=k^-_\delta}^{k^+_\delta}\left( \| \E[x_{k}]-x^\dagger\| + \|\E[x_{k}]-\E[x_{k}^\delta]\| + \| x_{k}^\delta - \E [x_{k}^\delta]\|\right)\chi_{\{k(\delta)=k\}}\\
                                      &\le \sum_{k=k^-_\delta}^{k^+_\delta} \left( c_1 (k-1)^{-(1-\alpha)p} + c_2 \delta (k-1)^\frac{1-\alpha}{2} + \|x_k^\delta -\E[x_k^\delta]\|\right)\chi_{\{k(\delta)=k\}}\\
                                      &\le c_1\left(k^-_\delta-1\right)^{-(1-\alpha)p} + c_2 \delta \left(k^+_\delta-1\right)^\frac{1-\alpha}{2} + \sum_{k=k^-_\delta}^{k^+_\delta} \| x_k^\delta - \E[x_k^\delta]\|\chi_{\{k(\delta)=k\}},
\end{align*}
where we have used \cite[Theorem 3.2]{JinLu:2019} and Lemma \ref{thm:JinLu} below in the third line. The first two terms
clearly tend to $0$ for $\delta\to 0^+$ (since $k_\delta^-\to\infty$, and
$\delta (k_\delta^+)^\frac{1-\alpha}{2}\to0$, in view of Theorem \ref{thm:dp}).
By Markov's inequality \cite[p. 242]{Feller:1968} and the independence assumption between $k(\delta)$ and $x_{k(\delta)}^\delta$,
\begin{align*}
   \mathbb{P}\left( \sum_{k=k^-_\delta}^{k^+_\delta} \| x_k^\delta-\E[x_k^\delta]\| \chi_{\{k(\delta)=k\}} >\varepsilon'\right)
   &\le \frac{\sum_{k=k^-_\delta}^{k^+_\delta}\E\left[ \| x_k^\delta-\E[x_k^\delta]\| \chi_{\{k(\delta)=k\}}\right]}{\varepsilon'}\\
   & = \frac{\sum_{k=k^-_\delta}^{k^+_\delta}\E\left[ \| x_k^\delta-\E[x_k^\delta]\|\right] \mathbb{P}\left(k(\delta)=k\right)}{\varepsilon'}.
\end{align*}
Now Jensen's inequality and Proposition
\ref{prop:sgd-residual} below (with $s=0$, $\gamma<\min(\alpha,1-\alpha)$ and $\beta<1-\alpha$) give
\begin{align*}
 \mathbb{P}\left( \sum_{k=k^-_\delta}^{k^+_\delta} \| x_k^\delta-\E[x_k^\delta]\| \chi_{\{k(\delta)=k\}} >\varepsilon'\right)
  &\le \frac{\sum_{k=k^-_\delta}^{k^+_\delta}\sqrt{\E\left[ \| x_k^\delta-\E[x_k^\delta]\|^2\right]} \mathbb{P}\left(k(\delta)=k\right)}{\varepsilon'}\\
  &\le \frac{\sqrt{c((k^-_\delta)^{-\beta}+\delta^2 (k^-_\delta)^{-\gamma})} \sum_{k=k^-_\delta}^{k^+_\delta}\mathbb{P}\left(k(\delta)=k\right)}{\varepsilon'}\\
  &= \frac{\sqrt{c((k^-_\delta)^{-\beta}+\delta (k^-_\delta)^{-\gamma})} \mathbb{P}\left(\Omega_\delta\right)}{\varepsilon'}\to 0
\end{align*}
as $\delta\to0^+$. Thus it follows that
\begin{equation*}
 \mathbb{P}\left( \|x_k(\delta)-x^\dagger\|>\varepsilon, k_\delta^-<k(\delta)\le k^+_\delta\right)\to 0
\end{equation*}
as $\delta\to0^+$. This completes the proof of the theorem.
\end{proof}
\begin{remark}\label{rm:thm2}
Clearly, with $k_\omega(\delta)$ given as in Remark \ref{rm:thm1}, there holds $\mathbb{P}\left(\|x_{k_\omega(\delta)}-x^\dagger\| \ge \varepsilon\right)\to 0$ for $\delta\to 0^+$. That is, the convergence remains valid for the variant of the discrepancy  principle \eqref{eqn:dp} evaluated with a frequency.
\end{remark}

\section{The proofs of Propositions \ref{prop:dp-bdd} and \ref{prop:sgd-variance}}\label{sec:auxil}
In this part, we prove Propositions \ref{prop:dp-bdd} and \ref{prop:sgd-variance}, which are used in the proof of the
Theorems \ref{thm:dp} and \ref{thm:dp:conv}. We shall use the following result from \cite[Theorem 3.1]{JinLu:2019}
frequently. Note that $\|B^\frac12 (x_k-x^\dagger)\| = \|Ax_k-y^\dagger\|/\sqrt{n}$.
\begin{lemma}\label{thm:JinLu}
Let Assumption \ref{ass:sgd} be fulfilled, then for $s\in\{0,\frac12\}$ and $c_{p,s}:=\left(\frac{
(p+s)(1-\alpha)}{c_0e(2^{1-\alpha}-1)}\right)^{p+s} \| w\|$, there holds
\begin{equation*}
\| B^s(x_{k+1}-x^\dagger)\| \le c_{p,s} k^{-(p+s)(1-\alpha)}.
\end{equation*}
\end{lemma}

\subsection{The proof of Proposition \ref{prop:dp-bdd}}
\begin{proof}
We may assume $k^*>2$. By the definition of $k^*(\delta)$ and the triangle inequality
\begin{align}
  \tau^* \delta &\le \|A \E[ x_{k^*-1}^\delta]-y^\delta\|\nonumber\\
   &\le \|A\E[x_{k^*-1}]-y^\dagger\| + \|A\E[x_{k^*-1}^\delta-x_{k^*-1}]+(y^\dagger-y^\delta)\|.\nonumber
\end{align}
By Lemma \ref{thm:JinLu}, the term $\|A\E[x_{k^*-1}]-y^\dag\|$ is bounded by
\begin{equation}\label{eqn:appbdd}
  \|A\E[x_{k^*-1}]-y^\dagger\| \leq  c_{p} (k^*-2)^{-(p+\frac12)(1-\alpha)},\quad \mbox{with  }
  c_p=\sqrt{n}c_{p,\frac12}.
\end{equation}
Next we claim
\begin{equation}\label{eqn:resbdd}
  \|A\E[x_{k^*-1}^\delta-x_{k^*-1}]+(y^\dagger-y^\delta)\|\leq \delta.
\end{equation}
Combining \eqref{eqn:appbdd} with \eqref{eqn:resbdd} immediately implies the desired assertion.
It remains to show the claim \eqref{eqn:resbdd}. To this end, we employ the filter of the Landweber
method. The relation (\ref{eqn:sgd=lm}) implies that $\E[x_k^\delta]$
satisfies the following recursion
\begin{align*}
  A\E[x_{k+1}^\delta]-y^\delta &= \left(I-\frac{\eta_k}{n}AA^t\right)\left(A\E[x_k^{\delta}]-y^\delta\right).
\end{align*}
Using this yields
\begin{equation}\label{eqn:dp-bdd0}
A\E[x_k^\delta]-y^\delta = \prod_{j=1}^{k-1}\left(I-\frac{\eta_j}{n}AA^t\right)\left(Ax_1-y^\delta\right),
\end{equation}
and consequently, by the choice of $c_0$,
\begin{equation}\label{eqn:dp-bdd}
 \|A\E[x_{k}^\delta-x_{k}]+(y^\dagger-y^\delta)\| =\left\| \prod_{j=1}^{k-1}\left(I-\frac{\eta_j}{n}AA^t\right)\left(y^\dagger-y^\delta\right)\right\|\le \delta.
\end{equation}
This completes the proof of the proposition.
\end{proof}

\subsection{Proof of Proposition \ref{prop:sgd-variance}}
The proof of Proposition \ref{prop:sgd-variance} employs several technical estimates \cite{JinLu:2019}.
\begin{lemma}\label{lem:estimate-B}
For any $j<k$, and any symmetric and positive semidefinite operator $S$ and stepsizes $\eta_j\in(0,\|S\|^{-1}]$ and $p\geq 0$, there holds
\begin{equation*}
  \|\prod_{i=j}^k(I-\eta_iS)S^p\|\leq \frac{p^p}{e^p(\sum_{i=j}^k\eta_i)^p}.
\end{equation*}
\end{lemma}


Next we recall two useful estimates taken from \cite{JinLu:2019}.
\begin{lemma}\label{lem:basicest1}
For $\eta_j=\eta_0j^{-\alpha}$ with $\alpha\in(0,1)$, $\beta\in[0,1]$ and $r\geq0$, there hold
\begin{align*}
  \sum_{j=1}^{[\frac k2]}\frac{\eta_j^2}{(\sum_{\ell=j+1}^k\eta_\ell)^r}j^{-\beta} &\leq c_{\alpha,\beta,r}k^{-r(1-\alpha)+\max(0,1-2\alpha-\beta)},\\
  \sum_{j=[\frac k2]+1}^{k-1}\frac{\eta_j^2}{(\sum_{\ell=j+1}^k\eta_\ell)^r}j^{-\beta}&\leq c'_{\alpha,\beta,r}k^{-((2-r)\alpha+\beta)+\max(0,1-r)},
\end{align*}
where we slightly abuse the notation $k^{-\max(0,0)}$ for $\ln k$, and $c_{\alpha,\beta,r}$ and $c'_{\alpha,\beta,r}$ are given by
\begin{align*}
   c_{\alpha,\beta,r} & = 2^{r}\eta_0^{2-r}\left\{\begin{array}{ll}
     \frac{2\alpha+\beta}{2\alpha+\beta-1}, & 2\alpha+\beta >1,\\
     2, & 2\alpha +\beta =1,\\
     \frac{2^{2\alpha+\beta-1}}{1-2\alpha-\beta}, & 2\alpha+\beta <1,
   \end{array}\right.\quad\mbox{and}\quad
   c'_{\alpha,\beta,r}  = 2^{2\alpha+\beta}\eta_0^{2-r}\left\{\begin{array}{ll}
     \frac{r}{r-1}, & r>1,\\
     2, & r= 1,\\
     \frac{2^{r-1}}{1-r}, & r<1.
   \end{array}\right.
\end{align*}
\end{lemma}

The next result gives an important recursion between the variance estimate.
\begin{lemma}\label{lem:recursion-var}
Let Assumption \ref{ass:sgd} be fulfilled. Then for the SGD iterate $x_k^\delta$, with
$\phi_j^s=\|B^{\frac{1}{2}+s}\Pi_{j+1}^k(B)\|$, there holds
\begin{align*}
  &\E[\|B^s(x_{k+1}^\delta-\E[x_{k+1}^\delta])\|^2]\\
  \leq  &\sum_{j=1}^k\eta_j^2(\phi_j^s)^2
  \left(c_s\E[\|B^s\left(x_j^\delta-\E[x_j^\delta]\right)\|^2]+2c_pj^{-2(1-\alpha)(p+\frac12)}+2\delta^2\right),
\end{align*}
with $s\in\{0,\frac12\}$ and $c_s,c_p$ given below.
\end{lemma}
\begin{proof}
By \cite[Theorem 3.3]{JinLu:2019} and the bias variance decomposition, the left hand side (LHS) is bounded by
\begin{align*}
  {\rm LHS} 
   \le & \sum_{j=1}^k\eta_j^2(\phi_j^s)^2 \E[\| Ax_j^\delta-y^\delta\|^2\\
    =& \sum_{j=1}^k\eta_j^2(\phi_j^s)^2 \left(\E[\| A\left(x_j^\delta-\E[x_j^\delta]\right)\|^2]+\|A\E[x_j^\delta]-y^\delta\|^2\right).
\end{align*}
Now by the triangle inequality  and \eqref{eqn:dp-bdd},
\begin{align*}
   {\rm LHS}\le &\sum_{j=1}^k\eta_j^2(\phi_j^s)^2 \left(\E[\| A\left(x_j^\delta-\E[x_j^\delta]\right)\|^2]+\left(\|A\E[x_j]-y^\dagger\| + \|A\left(\E[x_j^\delta]-\E[x_j]\right)-\left(y^\delta-y^\dagger\right)\|\right)^2\right)
   \end{align*}
Since $\|A \E [x_1]-y^\dagger\| = \| y^\dagger\|$, and
\begin{equation*}
\|A \E[x_j] - y^\dagger \| \le \sqrt{n}c_{p,\frac12} (j-1)^{-(p+\frac12)(1-\alpha)} \le \sqrt{n}c_{p,\frac12} 2^{(p+\frac12)(1-\alpha)}j^{-(p+\frac12)(1-\alpha)}
   \end{equation*}
for $j\ge 2$ by Lemma \ref{thm:JinLu}. Thus, with $c_p:= \left(\max\{\|y^\dagger \|, \sqrt{n}c_{p,\frac12} 2^{(p+\frac12)(1-\alpha)}\}\right)^2$,
\begin{align*}
{\rm LHS} \le&\sum_{j=1}^k\eta_j^2(\phi_j^s)^2
  \left(n^{2s}\| A \|^{4(\frac12-s)}\E[\|B^s\left(x_j^\delta-\E[x_j^\delta]\right)\|^2]+2c_pj^{-2(1-\alpha)(p+\frac12)}+2\delta^2\right)
\end{align*}
which completes the proof of the lemma with $c_s=n^{2s}\|A\|^{4(\frac12-s)}$.
\end{proof}

The next result gives a sharp estimate on $\E[\|B^s(x_k^\delta-\E[x_k^\delta])\|^2]$.
\begin{proposition}\label{prop:sgd-residual}
Let Assumption \ref{ass:sgd} be fulfilled. Then for the SGD iterate $x_k^\delta$, the mean squared error
$\E[\|B^s(x_k^\delta-\E[x_k^\delta])\|^2]$ with $s\in\{0,\frac12\}$ satisfies
\begin{equation*}
 \E[\|B^s(x_k^\delta-\E[x_k^\delta])\|^2] \leq c(\alpha,p,n,s,\beta,\gamma)(k^{-\beta} + \delta^2k^{-\gamma})
\end{equation*}

\noindent for $\beta<\min\left((1+2s)(1-\alpha),(1+2p)(1-\alpha)+\alpha\right)$ and $\gamma<\min(\alpha,1-\alpha)$.

\end{proposition}
\begin{proof}
 Lemma \ref{lem:recursion-var} implies that
the weighted mean squares error $d_j^s = \E[\|B^s(x_k^\delta-x^\dag)\|^2]$ satisfies the following recursion
\begin{equation}\label{eqn:variance-recur}
d_{k+1}^s\leq  \sum_{j=1}^k\eta_j^2(\phi_j^s)^2
  \left(c_sd_j^s+2c_pj^{-2(1-\alpha)(p+s)}+2\delta^2\right)
\end{equation}
Now we prove the desired assertion by mathematical induction (with $\beta=(2p+1)(1-\alpha)$):
\begin{equation*}
  d_k^s \leq c(k^{-\beta} + \delta^2k^{-\gamma}),
\end{equation*}
where the constant $c\geq 1$ is to be determined. This assertion holds trivially for all finite $k$, up to $k^*$,
provided that $c$ is sufficiently large. Now suppose the assertion holds for $k\geq k^*$, and we
prove the assertion for $k+1$. Indeed, it follows from the recursion \eqref{eqn:variance-recur}, the induction hypothesis and since
$\beta<2(1-\alpha)(p+\frac12)$, that
\begin{align*}
  d_{k+1}^s & \leq \sum_{j=1}^k\eta_j^2(\phi_j^s)^2(c_sc(j^{-\beta} + j^{-\gamma}\delta^2) +2c_p j^{2(1-\alpha)(p+s)}+2\delta^2)\\
          & \leq c_sc\sum_{j=1}^k\eta_j^2(\phi_j^s)^2j^{-\beta} +(c_sc+2)\delta^2\sum_{j=1}^k\eta_j^2(\phi_j^s)^2 +2c_p \sum_{j=1}^k\eta_j^2(\phi_j^s)^2 j^{-2(1-\alpha)(p+\frac12)}\\
          &\le (c_sc+2c_p)\sum_{j=1}^k\eta_j^2(\phi_j^s)^2j^{-\beta'} +(c_sc+2)\delta^2\sum_{j=1}^k\eta_j^2(\phi_j^s)^2.
\end{align*}
\noindent with $\beta'=\min(\beta,(1+2p)(1-\alpha))$. Without loss of generality, we may assume that
$\beta'\ge 1-2\alpha$. By Lemmas \ref{lem:estimate-B} and \ref{lem:basicest1}, the first sum is bounded by
\begin{align}
    \sum_{j=1}^k \eta_j^2(\phi_j^s)^2 j^{-\beta'} \leq& e^{-2}c_{\alpha,\beta',1+2s}k^{-(1+2s)(1-\alpha)+\max(0,1-2\alpha-\beta')}\nonumber\\
      &\quad + e^{-1}c'_{\alpha,\beta',1}\|B\|k^{-(\alpha+\beta')}\ln k
       + c_0^2\|B\|^2k^{-(2\alpha+\beta')}.\label{eqn:est-recur-1}
\end{align}
Since $\beta'+\alpha>\beta$ and $ \max(0,1-2\alpha-\beta')=0$, thus,
\begin{align*}
    \sum_{j=1}^k \eta_j^2\phi_j^2 j^{-\beta'} \leq&
     (e^{-2}c_{\alpha,\beta',1+2s} k^{-(1+2s)(1-\alpha)+\beta}\ln k + e^{-1}c'_{\alpha,\beta',1}\|B\|k^{-(\alpha+\beta')+\beta }\ln k
       + c_0^2\|B\|^2k^{-\alpha})k^{-\beta}.
\end{align*}
Meanwhile, with $-(1+2s)(1-\alpha)+\max(0,1-2\alpha)=-\min((1+2s)(1-\alpha),\alpha+2s(1-\alpha))$, we obtain
\begin{align*}
  \sum_{j=1}^k\eta_j^2(\phi_j)^2 &\leq  e^{-2}c_{\alpha,0,2} k^{-\min((1+2s)(1-\alpha),\alpha+2s(1-\alpha))} + e^{-1}c'_{\alpha,0,1}\|B\|k^{-\alpha }\ln k
       + c_0^2\|B\|^2k^{-2\alpha}\\
         & \leq (e^{-2}c_{\alpha,0,1+2s}k^{-\min((1-\alpha),\alpha)+\gamma} + e^{-1}c'_{\alpha,0,1}\|B\|k^{-\alpha+\gamma}\ln k  + c_0^2\|B\|^2k^{-2\alpha+\gamma})k^{-\gamma}
\end{align*}
Combining the preceding estimates yields
\begin{align*}
  d_{k+1} &\leq (cc_s+2c_p)\left(e^{-2}c_{\alpha,\beta',1+2s} k^{-(1+2s)(1-\alpha)+\beta}\ln k + e^{-1}c'_{\alpha,\beta',1}\|B\|k^{-(\alpha+\beta')+\beta }\ln k
       + c_0^2\|B\|^2k^{-\alpha}\right)k^{-\beta}\\
   &\quad + (c_sc+2) \delta^2\left(e^{-2}c_{\alpha,0,1+2s}k^{-\min((1-\alpha),\alpha)+\gamma} + e^{-1}c'_{\alpha,0,1}\|B\|k^{-\alpha+\gamma}\ln k  + c_0^2\|B\|^2k^{-2\alpha+\gamma}\right)k^{-\gamma}.
\end{align*}
Since by assumption, $\beta<(1+2s)(1-\alpha)$, $\beta<\alpha+\beta'$ and $\gamma<\min(\alpha,1-\alpha)$, there exists $k^*$ such that for all $k\geq k^*$
\begin{align*}
  (c_s+2c_p)\left(e^{-2}c_{\alpha,\beta',1+2s} k^{-(1+2s)(1-\alpha)+\beta}\ln k + e^{-1}c'_{\alpha,\beta',1}\|B\|k^{-(\alpha+\beta')+\beta }\ln k
       + c_0^2\|B\|^2k^{-2\alpha}\right)<\tfrac14,\\
 (c_s+2) \delta^2\left(e^{-2}c_{\alpha,0,1+2s}k^{-\min((1-\alpha),\alpha)+\gamma} + e^{-1}c'_{\alpha,0,1}\|B\|k^{-\alpha+\gamma}\ln k  + c_0^2\|B\|^2k^{-2\alpha+\gamma}\right) <\tfrac14.
\end{align*}
Thus, with this choice of $k^*$ and $k\ge k^*$,
\begin{align*}
d_{k+1} &\le \tfrac{c}{4}\left( k^{-\beta}+\delta^2k^{-\gamma}\right) \le c \tfrac{(1+k^{-1})^{\beta}}{4}\left( (k+1)^{-\beta}+\delta^2 (k+1)^{-\gamma}\right)<c\left((k+1)^{-\beta}+\delta^2 (k+1)^{-\gamma}\right)
\end{align*}
and we obtain the desired assertion.
\end{proof}

\begin{remark}
The $n$ factor in the estimate is due to the variance inflation of using stochastic
gradients in place of gradient in SGD. This factor can be reduced by suitable variance
reduction techniques, e.g., mini-batching and stochastic variance reduced gradient
\cite{JohnsonZhang:2013}. Note that with \cite[Theorems 3.1 and 3.2]{JinLu:2019} and $s=0$, Proposition
\ref{prop:sgd-residual} gives an improved (regarding the exponents) a priori bound for
the mean squared error $\E[\|x_k^\delta-x^\dagger\|^2]$.
\end{remark}

Last, using Lemma \ref{lem:recursion-var} and Proposition \ref{prop:sgd-residual}, we can prove Proposition
\ref{prop:sgd-variance}.
\begin{proof}[Proof of Proposition \ref{prop:sgd-variance}]
Using Lemma \ref{lem:recursion-var} and Proposition \ref{prop:sgd-residual} with $s=\frac12$ and $c=c(\alpha,p,n,s,\beta,\gamma)$, we deduce
\begin{align*}
  \E[\|A(x_{\kappa(\delta)}^\delta-\E[x_{\kappa(\delta)}^\delta])\|^2] & \le nc\left( \kappa(\delta)^{-\beta}+\delta^2 \kappa(\delta)^{-\gamma}\right).
\end{align*}
We choose $\gamma>0$. If $p<\frac12$ and $r>2p$, then we can choose $\beta>(1-\alpha)(2p+1)$, so with the choice $\kappa(\delta)=\delta^{-\frac{2}{(1-\alpha)(2p+1)}}$,
the claim follows. Otherwise, if $p\ge\frac12$, then we can choose $\beta>(1-\alpha)(r+1)$, so with the choice $\kappa(\delta)=\delta^{-\frac{2}{(1-\alpha)(r+1)}}$ the claim again follows. This completes the proof of the proposition.
\end{proof}

\section{Numerical experiments and discussions}\label{sec:numer}

Now we provide numerical experiments to complement the theoretical analysis. Three model examples, i.e., \texttt{phillips} (mildly
ill-posed, smooth), \texttt{gravity} (severely ill-posed, medium smooth)
and \texttt{shaw} (severely ill-posed, nonsmooth), are taken from the open source \texttt{MATLAB} package Regutools \cite{Hansen:2007}, available at
\url{http://people.compute.dtu.dk/pcha/Regutools/} (last accessed on April 14, 2020). The problems cover a variety of
setting, e.g., different solution smoothness and degree of ill-posedness. These examples are discretizations
of Fredholm/Volterra integral equations of the first kind, by means of either the Galerkin approximation with piecewise constant basis functions or quadrature rules.
All the examples are discretized into a linear system of size $n=m=1000$. In addition, we generate a synthetic example, termed
\texttt{smoothed-phillips}, whose exact solution $x^\dag$ is first generated by $\bar x^\dag = A^tAA^t \bar y^\dag$
and then normalized to have unit maximum, i.e., $x^\dag=\bar x^\dag/\|\bar x^\dag\|_{\ell^\infty}$, where A is the
system matrix and $\bar y^\dag$ the exact data from \texttt{phillips}, and the corresponding exact data is formed
by $ y^\dag = Ax^\dag $. By its very construction, the solution $x^\dag$ satisfies Assumption \ref{ass:sgd}(ii) with an exponent
$p>2$, and thus it is very smooth in some sense. Throughout, the noisy data $y^\delta$ is generated according to
\begin{equation*}
  y_i^\delta: = y_i^\dagger + \delta \max_{j}(| y^\dagger_j|)\xi_i,\quad i=1,\ldots,n,
\end{equation*}
where the i.i.d. random variables $\xi_i$ follow the standard Gaussian distribution (with zero mean and unit variance), and $\delta>0$ denotes
the relative noise level (by slightly abusing the notation). The parameter $c_0$ in the stepsize schedule in Assumption \ref{ass:sgd}(i)
is set to $(\max_i\|a_i\|^2)^{-1}$, the exponent $\alpha$ is taken from the set $\{0.1,0.3,0.5\}$, and unless otherwise stated, the stopping
criterion is tested every 100 SGD iterations (see Remarks \ref{rm:thm1} and \ref{rm:thm2}). SGD is always initialized with $x_1=0$,
and the maximum number of epochs is fixed at $5000$, where one epoch refers to $n$ SGD iterations. The parameter $\tau$ in the
discrepancy principle \eqref{eqn:dp} is fixed at $\tau=1.2$. All the statistical quantities presented below are computed from 100
independent runs.

\subsection{Optimality}
{First, we verify the optimality of the discrepancy principle \eqref{eqn:dp}, against an order
optimal regularization method. There are many possible choices, e.g., Landweber method and conjugate
gradient method \cite[Chapters 6 and 7]{EnglHankeNeubauer:1996}. In this work, we employ the Landweber method
as the benchmark. The Landweber method generally converges steadily although
often slowly.} However, it is known to be an order optimal regularization method with
infinite qualification \cite[Theorem 6.5, p. 159]{EnglHankeNeubauer:1996}, when terminated by the discrepancy
principle \eqref{eqn:dpex}, {and further, it is the population version of SGD}
(the expected iterates $\left(\E[x_k^\delta]\right)_{k
\in\mathbb{N}}$ are exactly the Landweber iterates; see \eqref{eqn:sgd=lm}), and thus it serves a
good benchmark for performance comparison in terms of the convergence rate. For the comparison, the Landweber
method is initialized with $x_1=0$, with a constant stepsize $1/\|A\|^2$, and it is terminated with the discrepancy principle
\eqref{eqn:dpex} with $\tau^*=1.2$ (i.e., the same as for SGD) with the maximum number of iterations being fixed at 5000.
The numerical results for the examples are summarized in Tables \ref{tab:lm-phillips}--\ref{tab:lm-phil-smooth}. In the tables,
$e_{\rm sgd}$ and $\rm{std}(e_{\rm sgd})$ denote the {(sample)} mean and the {(sample)} standard deviation of the (squared) error
$\|x_{k_\delta}^\delta-x^\dag\|^2$, respectively, i.e.,
\begin{equation*}
 e_{\rm sgd}=\E[\|x_{k_\delta}^\delta-x^\dag\|^2]\quad \mbox{and}\quad  {\rm std}(e_{\rm sgd}) =
 \E[(\|x_{k^\delta}^\delta-x^\dag\|^2-e_{\rm sgd})^2]^\frac{1}{2},
\end{equation*}
and $k_{\rm sgd}=\E[k_{\delta}]$ is the mean stopping index for SGD, in terms of the number of epochs. Likewise $e_{\rm lm}$ and $k_{\rm lm}$
denote the squared reconstruction error and stopping index, respectively, of the Landweber method, terminated according to
the discrepancy principle \eqref{eqn:dpex}.

\begin{table}[hbt!]
\centering
\caption{Comparison between SGD and LM for \texttt{phillips}.\label{tab:lm-phillips}}
\setlength{\tabcolsep}{4pt}
\begin{tabular}{c|ccccccccc|cc|}
\toprule
\multicolumn{1}{c}{}&\multicolumn{3}{c}{$\alpha=0.1$} & \multicolumn{3}{c}{$\alpha=0.3$} & \multicolumn{3}{c}{$\alpha=0.5$} & \multicolumn{2}{c}{LM} \\
\cmidrule(l){2-4} \cmidrule(l){5-7} \cmidrule(l){8-10} \cmidrule(l){11-12}
$\delta$ & $e_{\rm sgd}$ & $\mathrm{std}(e_{\rm sgd})$ & $k_{\rm sgd}$ & $e_{\rm sgd}$ & $\mathrm{std}(e_{\rm sgd})$ & $k_{\rm sgd}$ & $e_{\rm sgd}$ & $\mathrm{std}(e_{\rm sgd})$ & $k_{\rm sgd}$ & $e_{\rm lm}$ & $k_{\rm lm}$\\
1e-3 & 8.60e-3 &  4.53e-3 &  1.424 &  8.53e-3 &  4.42e-3 &  4.189 &   8.34e-3 &  4.60e-3 &  52.29 &  5.72e-3 &  361 \\
5e-3 & 1.70e-2 &  8.41e-3 &  0.458 &  2.31e-2 &  8.81e-3 &  0.975 &   2.48e-2 &  7.38e-3 &  6.032 &  2.26e-2 &  128 \\
1e-2 & 2.82e-2 &  1.62e-2 &  0.281 &  4.72e-2 &  2.07e-2 &  0.433 &   5.78e-2 &  2.04e-2 &  1.647 &  5.76e-2 &  51 \\
5e-2 & 1.41e-1 &  9.70e-2 &  0.157 &  1.49e-1 &  9.01e-2 &  0.116 &   2.11e-1 &  9.69e-2 &  0.173 &  2.19e-1 &  15 \\
\bottomrule
\end{tabular}
\end{table}

\begin{table}[hbt!]
\centering
\caption{Comparison between SGD and LM for \texttt{gravity}.\label{tab:lm-gravity}}
\setlength{\tabcolsep}{4pt}
\begin{tabular}{c|ccccccccc|cc|}
\toprule
\multicolumn{1}{c}{}&\multicolumn{3}{c}{$\alpha=0.1$} & \multicolumn{3}{c}{$\alpha=0.3$} & \multicolumn{3}{c}{$\alpha=0.5$} & \multicolumn{2}{c}{LM} \\
\cmidrule(l){2-4} \cmidrule(l){5-7} \cmidrule(l){8-10} \cmidrule(l){11-12}
$\delta$ & $e_{\rm sgd}$ & $\mathrm{std}(e_{\rm sgd})$ & $k_{\rm sgd}$ & $e_{\rm sgd}$ & $\mathrm{std}(e_{\rm sgd})$ & $k_{\rm sgd}$ & $e_{\rm sgd}$ & $\mathrm{std}(e_{\rm sgd})$ & $k_{\rm sgd}$ & $e_{\rm lm}$ & $k_{\rm lm}$\\
1e-3 & 6.71e-1 &  2.61e-1 &  1.960  &  7.46e-1 &  2.73e-1 &  9.316 &  7.78e-1 &  2.49e-1 &  198.5 &  7.25e-1 &  640 \\
5e-3 & 2.00e0  &  8.91e-1 &  0.451  &  2.53e0  &  1.12e0  &  0.880 &  2.76e0  &  1.14e0  &  6.217 &  2.44e0  &  95  \\
1e-2 & 3.12e0  &  1.57e0  &  0.250  &  4.33e0  &  1.92e0  &  0.361 &  4.74e0  &  2.07e0  &  1.366 &  4.02e0  &  50 \\
5e-2 & 9.07e0  &  5.31e0  &  0.143  &  1.15e1  &  6.61e0  &  0.107 &  1.52e1  &  7.46e0  &  0.135 &  1.66e1  &  9  \\
\bottomrule
\end{tabular}
\end{table}

\begin{table}[hbt!]
\centering
\caption{Comparison between SGD and LM for \texttt{shaw}.\label{tab:lm-shaw}}
\setlength{\tabcolsep}{4pt}
\begin{tabular}{c|ccccccccc|cc|}
\toprule
\multicolumn{1}{c}{}&\multicolumn{3}{c}{$\alpha=0.1$} & \multicolumn{3}{c}{$\alpha=0.3$} & \multicolumn{3}{c}{$\alpha=0.5$} & \multicolumn{2}{c}{LM} \\
\cmidrule(l){2-4} \cmidrule(l){5-7} \cmidrule(l){8-10} \cmidrule(l){11-12}
$\delta$ & $e_{\rm sgd}$ & $\mathrm{std}(e_{\rm sgd})$ & $k_{\rm sgd}$ & $e_{\rm sgd}$ & $\mathrm{std}(e_{\rm sgd})$ & $k_{\rm sgd}$ & $e_{\rm sgd}$ & $\mathrm{std}(e_{\rm sgd})$ & $k_{\rm sgd}$ & $e_{\rm lm}$ & $k_{\rm lm}$\\
1e-3 & 8.29e0 &  9.35e-2 & 57.73 &  8.47e0 &  5.59e-2 & 891.3 &  2.01e1 &  5.64e-1 & 5000  &  1.28e1 &  5000\\
5e-3 & 2.77e1 &  1.24e0  & 0.948 &  2.80e1 &  1.16e0  & 3.811 &  2.82e1 &  1.02e0  & 51.69 &  2.81e1 &  189 \\
1e-2 & 2.96e1 &  1.65e0  & 0.597 &  3.10e1 &  1.14e0  & 1.938 &  3.12e1 &  1.08e0  & 19.71 &  3.11e1 &  117 \\
5e-2 & 5.02e1 &  1.08e1  & 0.155 &  6.07e1 &  8.08e0  & 0.250 &  6.70e1 &  7.41e0  & 0.818 &  6.85e1 &  22  \\
\bottomrule
\end{tabular}
\end{table}

\begin{table}[hbt!]
\centering
\caption{Comparison between SGD and LM for \texttt{smoothed-phillips}.\label{tab:lm-phil-smooth}}
\setlength{\tabcolsep}{4pt}
\begin{tabular}{c|ccccccccc|cc|}
\toprule
\multicolumn{1}{c}{}&\multicolumn{3}{c}{$\alpha=0.1$} & \multicolumn{3}{c}{$\alpha=0.3$} & \multicolumn{3}{c}{$\alpha=0.5$} & \multicolumn{2}{c}{LM} \\
\cmidrule(l){2-4} \cmidrule(l){5-7} \cmidrule(l){8-10} \cmidrule(l){11-12}
$\delta$ & $e_{\rm sgd}$ & $\mathrm{std}(e_{\rm sgd})$ & $k_{\rm sgd}$ & $e_{\rm sgd}$ & $\mathrm{std}(e_{\rm sgd})$ & $k_{\rm sgd}$ & $e_{\rm sgd}$ & $\mathrm{std}(e_{\rm sgd})$ & $k_{\rm sgd}$ & $e_{\rm lm}$ & $k_{\rm lm}$\\
1e-3 &  1.63e-1 &  6.87e-2 &  1.348  & 1.59e-1 &  5.88e-2 &  4.030 &  1.55e-1 &  6.09e-2 &  48.02 &  1.51e-3 &  29 \\
5e-3 &  3.92e-1 &  2.08e-1 &  0.367  & 5.06e-1 &  2.05e-1 &  0.591 &  4.92e-1 &  1.99e-1 &  2.683 &  1.38e-2 &  18 \\
1e-2 &  5.95e-1 &  2.64e-1 &  0.242  & 8.57e-1 &  3.73e-1 &  0.303 &  9.46e-1 &  3.93e-1 &  0.774 &  4.06e-2 &  15 \\
5e-2 &  2.98e0  &  1.44e0  &  0.163  & 3.20e0  &  1.51e0  &  0.107 &  4.35e0  &  2.13e0  &  0.130 &  7.19e-1 &  9  \\
\bottomrule
\end{tabular}
\end{table}

The numerical results allow drawing a number of interesting observations. First, the exponent $\alpha$ in the stepsize
schedule exerts a strong influence on the (expected) stopping index $k_{\rm sgd}$. At low noise levels (i.e., small $\delta$), $k_{\rm
sgd}$ increases dramatically with the value of $\alpha$. Meanwhile, for any fixed $\alpha$, the error $e_{\rm sgd}$ increases
steadily with the noise level $\delta$, exhibiting the convergence behavior indicated in Theorem \ref{thm:dp:conv}.
Further, for each fixed $\delta$, the error $e_{\rm sgd}$ is largely comparable for all different $\alpha$ values,
although $k_{\rm sgd}$ increases with $\alpha$. This behavior is qualitatively in good agreement with Theorem \ref{thm:dp}: the upper bound scales as $O(\delta^{-\frac{2}{(1-\alpha)(\min\left(2p,r\right)+1)}})$. Thus, in practice, in order to obtain relatively efficient SGD, one
prefers small $\alpha$ values. Second, in terms of accuracy (measured by the mean squared error), SGD is competitive with the
classical Landweber method for \texttt{phillips}, \texttt{gravity} and \texttt{shaw}: $e_{\rm sgd}$
and $e_{\rm lm}$ are fairly close to each other in most cases, and $e_{\rm sgd}$ can be smaller than $e_{\rm lm}$, which fully confirms the
order-optimality of the discrepancy principle \eqref{eqn:dp} for SGD for low regularity solutions, and
also confirming the convergence in Theorem \ref{thm:dp:conv}. In fact, empirically, the error seems to converge not only in probability, but also in $L^2$.
A close inspection on the stopping index $k_{\rm sgd}$ is very telling: when the noise level $\delta$ is medium to large, the stopping index $k_{\rm sgd}$
of SGD, determined by \eqref{eqn:dp}, is ten-fold smaller than that for the Landweber method in terms of epoch
count. In particular, when the noise level $\delta$ is relatively high, SGD can actually deliver an accurate solution within less than one epoch,
i.e., going through only a fraction of all the available data points. Thus, in this regime, SGD is much more efficient than the Landweber method.
These observations are valid for all the examples, despite their dramatic difference in degree of ill-posedness and solution smoothness.
However, for \texttt{smoothed-phillips}, the achieved accuracy by SGD is far below than that by the Landweber method for all three
exponents $\alpha$. This suboptimality in convergence rate is attributed to the saturation phenomenon for SGD, due to the dominance
of the computational variance, when the true solution $x^\dag$ is very smooth. The effect of the variance component will be examined
more closely below in Section \ref{ssec:variance}.

The example \texttt{shaw} is challenging for numerical recovery, since the solution is far less smooth, and at low noise
level $\delta=$1e-3, the discrepancy principle \eqref{eqn:dpex} cannot be reached even after 5000 Landweber iterations,
see Table \ref{tab:lm-shaw}. A similar behavior is also
observed for SGD with $\alpha=0.3$ and $\alpha=0.5$. Nonetheless, with $\alpha=0.1$, the discrepancy principle \eqref{eqn:dp} can be
reached by SGD after a few hundred epochs, clearly showing the surprisingly beneficial effect of SGD noise for low-regularity
solutions.

Next we examine more closely the performance of individual samples. The boxplots are shown in Fig. \ref{fig:lm}
for the examples at two different scenarios, i.e., fixed $\alpha$ and fixed $\delta$. On each box, the central mark indicates the median, and the bottom
and top edges of the box indicate the 25th and 75th percentiles, respectively; The whiskers extend to the most
extreme data points not considered outliers, and the outliers are plotted individually using the '+' symbol. It is observed that for
a fixed $\alpha$, on average the error $\|x_{k(\delta)}^\delta-x^\dag\|^2$ increases with the noise level $\delta$ samplewise, and also its distribution broadens.
However, the required number of iterations to fulfill the discrepancy principle \eqref{eqn:dp} decreases dramatically, as the noise
level $\delta$ increases, concurring with the preceding observation that SGD is especially efficient for data with high noise levels. Meanwhile,
with the noise level $\delta$ fixed, the value of $\alpha$ does not change the results much overall. However, a larger $\alpha$ can potentially
make the percentile box larger and also more outliers, as shown by the results for \texttt{gravity} in Fig. \ref{fig:lm}, and
thus give less accurate results. This observation is counter-intuitive in that smaller variance does not immediately lead to better
accuracy. This might be related to the delicate interplay between the total error and various problem / algorithmic parameters, e.g., $\alpha$ and $p$.
Further, the outliers in the boxplots mostly lie above the box. These observations are typical for all the examples.

\begin{figure}[hbt!]
\centering
\includegraphics[width=.85\textwidth,trim={1.5cm 0.3cm 2cm 0.5cm}]{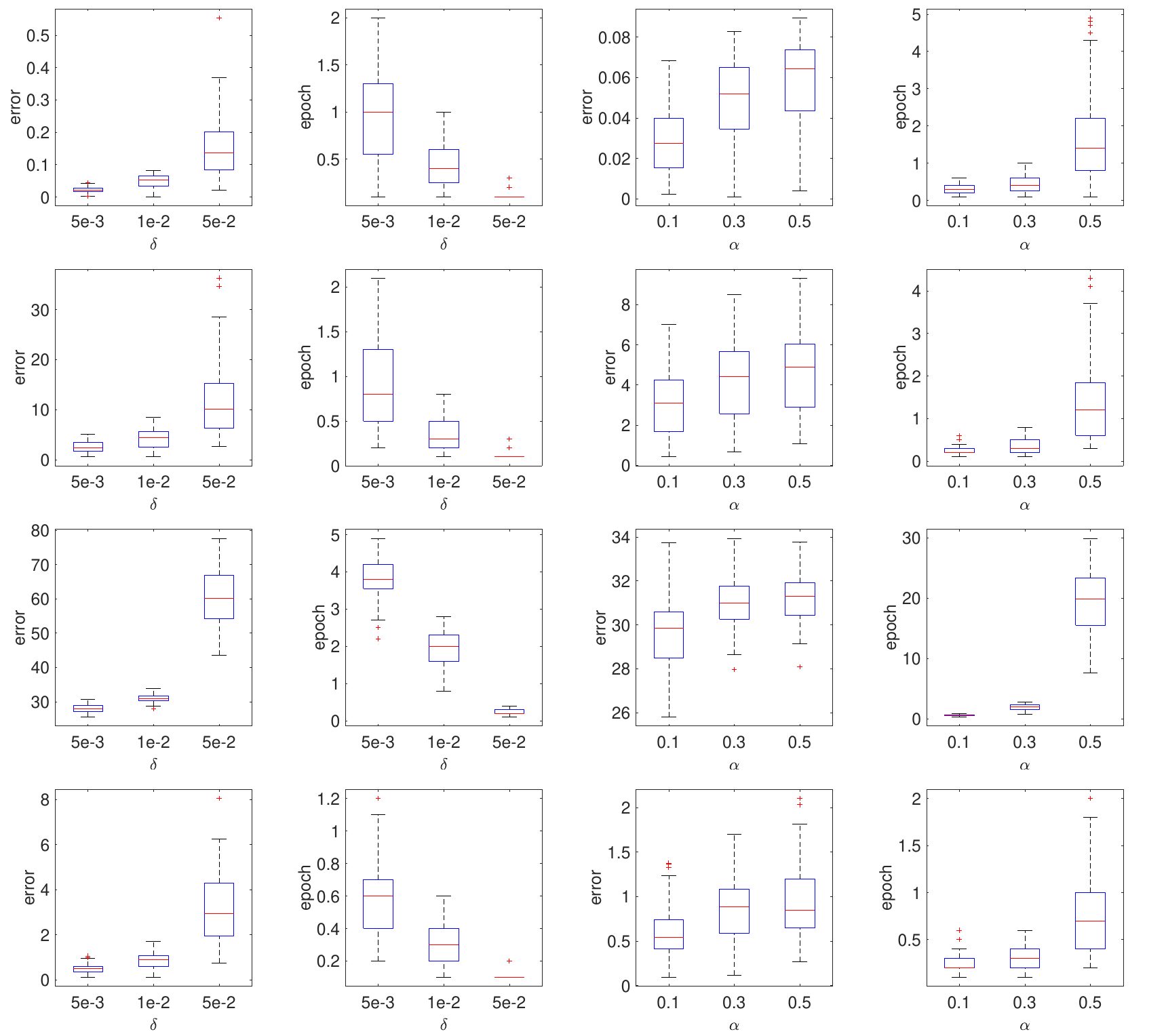}
\caption{Box plots for the error $\|x_{k_\delta}^\delta-x^\dag\|^2$ and the stopping index $k_\delta$ by SGD.
  The first two columns are obtained by SGD with $\alpha=0.3$, whereas the last two columns are for the noise level $\delta=$1e-2.
  The rows from top to bottom refer to \texttt{phillips}, \texttt{gravity}, \texttt{shaw} and \texttt{smoothed-phillips},
  respectively.}\label{fig:lm}
\end{figure}

\subsection{How influential is the variance?}\label{ssec:variance}

Now we examine more closely the dynamics of the SGD iteration via the bias-variance decomposition of the
error $\E[\|x_k^\delta-x^\dag\|^2]$ and residual $\E[\|Ax_k^\delta-y^\delta\|^2]$:
\begin{align*}
  \E[\|x_k^\delta-x^\dag\|^2] &= \|\E[x_k^\delta]-x^\dag\|^2 + \E[\|x_k^\delta-\E[x_k^\delta]\|^2],\\
  \E[\|Ax_k^\delta-y^\delta\|^2] &=\|A\E[x_k^\delta]-y^\delta\|^2 + \E[\|A(x_k^\delta-\E[x_k^\delta])\|^2].
\end{align*}
In Fig. \ref{fig:decomp}, we display the dynamics of mean squared error $\E[\|x_k^\delta-x^\dag\|^2]$ and
the mean squared residual $\E[\|Ax_k^\delta-y^\delta\|^2]$ together with their variance components for the
examples at two different relative noise levels, i.e., $\delta=$5e-3 and $\delta=$5e-2. At each time, SGD is run for
100 epochs (i.e., 1e5 SGD iterations), and the results are recorded every 50 SGD iterations, starting from
the 50th SGD iterations.

In the plots, we have indicated the true noise $\| y^\delta - y^\dagger\|^2$, also denoted by $\delta^2$. It is observed
that both $\E[\|x_k^\delta-x^\dag\|^2]$ and $\E[\|Ax_k^\delta-y^\delta\|^2]$ decay steadily at an algebraic rate
up to a value comparable to  the stopping index $k^*(\delta)$ for the Landweber method (by the discrepancy principle
\eqref{eqn:dpex}). Beyond the critical threshold $k^*(\delta)$, the error $\E[\|x_k^\delta-x^\dag\|^2]$ exhibits a
semiconvergence behavior in that it starts to increase, whereas the residual $\E[\|Ax_k^\delta-y^\delta\|^2]$ nearly
levels off at a value comparable with the noise level $\delta^2$ (actually it oscillates slightly, since the SGD iterate
is only descent for the residual on average). This is typical for iterative regularization methods
for inverse problems, since for the later iterates, the noise becomes the dominating driving force. Proposition
\ref{prop:sgd-residual} with $s=\frac12$ indicates that a similar behavior holds also for their variance components
(up to slightly beyond $k^*(\delta)$). Actually, the residual variance $\E[\|A(x_k^\delta-\E[x_k^\delta])\|^2]$ first
decays as $O(k^{-2(1-\alpha)})$ (upon ignoring the $\delta$ term), which matches well the empirical rate in the plot.
For the later iterates, as suggested by the $\delta$ term in Proposition \ref{prop:sgd-residual}, the decay is roughly
$O(k^{-\alpha})$. Likewise, the error variance $\E[\|x_k^\delta-\E[x_k^\delta]\|^2]$ decays slower at a rate
$O(k^{-(1-\alpha)})$. Interestingly, the decay rates of $\E[\|A(x_k^\delta-\E[x_k^\delta])\|^2]$ and $\E[\|x_k^\delta-
\E[x_k^\delta]\|^2]$ in the first and last columns are largely comparable, despite their drastic difference in the
smoothness of the exact solution $x^\dag$. Thus, the decay estimate in Proposition \ref{prop:sgd-residual} is actually
quite sharp, partially explaining the saturation phenomenon observed earlier. This behavior is consistently observed
for all three $\alpha$ values. It is worth noting that for \texttt{smoothed-phillips},
the curves for $\E[\|x_k^\delta-\E[x_k^\delta]\|^2]$ and $\E[\|x_k^\delta-x^\dag\|^2]$ nearly overlay
each other, i.e., the bias component is negligible after the initial 50 iterations, due to high smoothness of the
true solution, clearly indicating the saturation. For the other three examples, empirically, the variance components are of smaller order right after the initial 50 iterations. In particular, as stated in Proposition \ref{prop:sgd-variance}, $\E[\|A(x_k^\delta-\E[x_k^\delta])\|^2]$  contributes
very little to the mean squared residual $\E[\|Ax_k^\delta-y^\delta\|^2]$ in the neighborhood of $k^*(\delta)$. This occurs
for all three values of the exponent $\alpha$ in the stepsize schedule.
The observations hold also for individual realizations; see
Fig. \ref{fig:decomp0} for the corresponding plots. The overall behavior of the curves in Fig. \ref{fig:decomp0}
is fairly similar to that in Fig. \ref{fig:decomp}, except that the residual and error curves exhibit pronounced
oscillations due to the randomness of the row index selection. Nonetheless, in the neighborhood of $k^*(\delta)$, the variance components remain
much smaller in magnitude. This observation provides the key insight for the analysis in Section \ref{sec:bound}.

\begin{figure}
\centering
\includegraphics[width=0.85\textwidth,trim={1.5cm 0.3cm 2cm 0.5cm}]{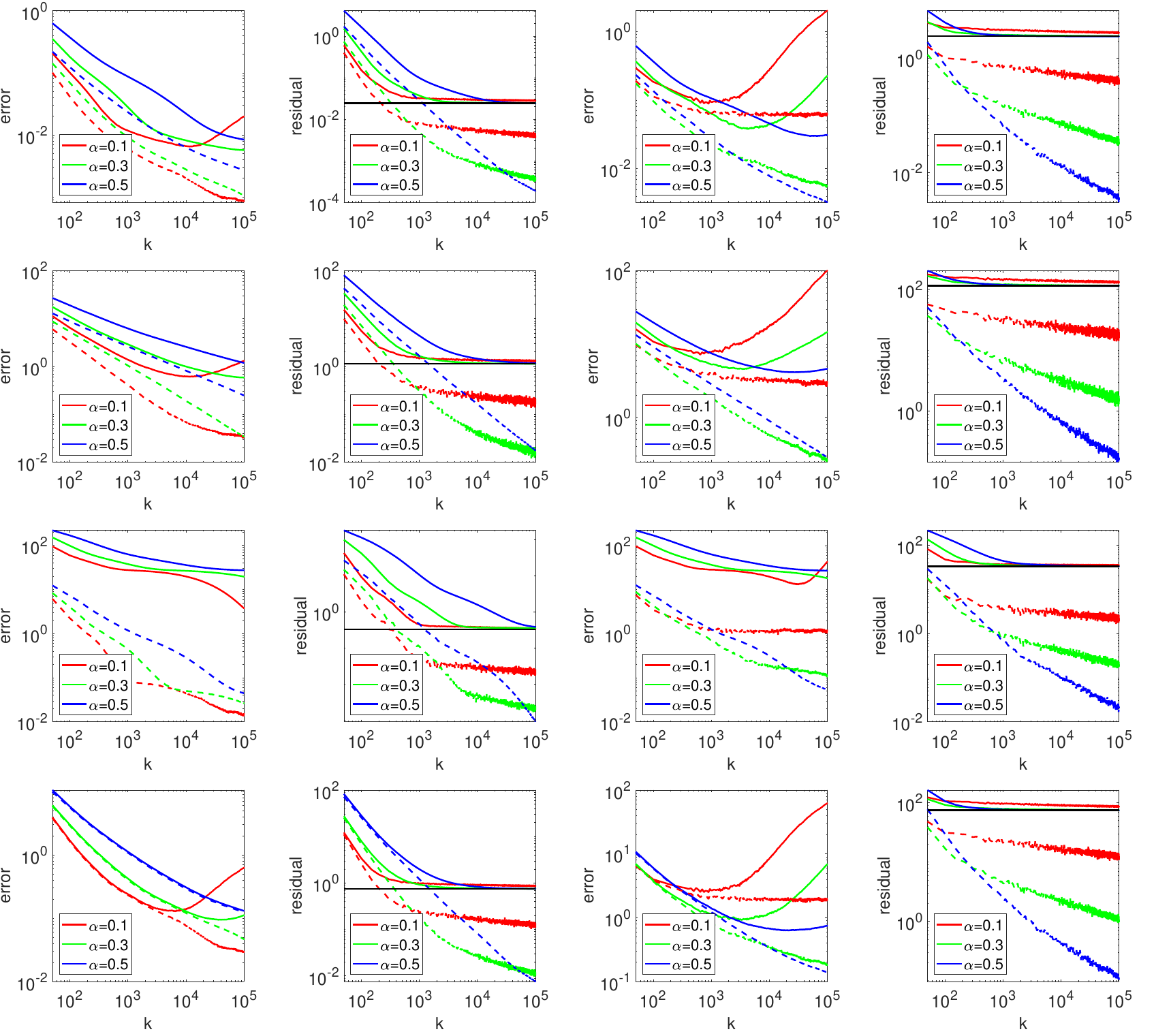}
\caption{The decay of the mean squared error $\E[\|x_k^\delta-x^\dag\|^2]$ and residual $\E[\|Ax_k^\delta-y^\delta\|^2]$
and their variance components $\E[\|x_k^\delta-\E[x_k^\delta]\|^2]$ and $\E[\|A(x_k^\delta-\E[x_k^\delta])\|^2]$ versus
the SGD iteration number $k$. The solid and dashed curves denote the mean squared quantity and the variance component,
respectively, and the black curve indicates the discrepancy $\delta^2=\|y^\delta-y^\dag\|^2$. The first two columns are
for the noise level $\delta=\mbox{5e-3}$ and the last two columns are for the noise level $\delta=\mbox{5e-2}$. The rows
from top to bottom refer to \texttt{phillips}, \texttt{gravity}, \texttt{shaw} and \texttt{smoothed-phillips}, respectively.}\label{fig:decomp}
\end{figure}

\begin{figure}
\centering
\includegraphics[width=0.85\textwidth,trim={1.5cm 0.3cm 2cm 0.5cm}]{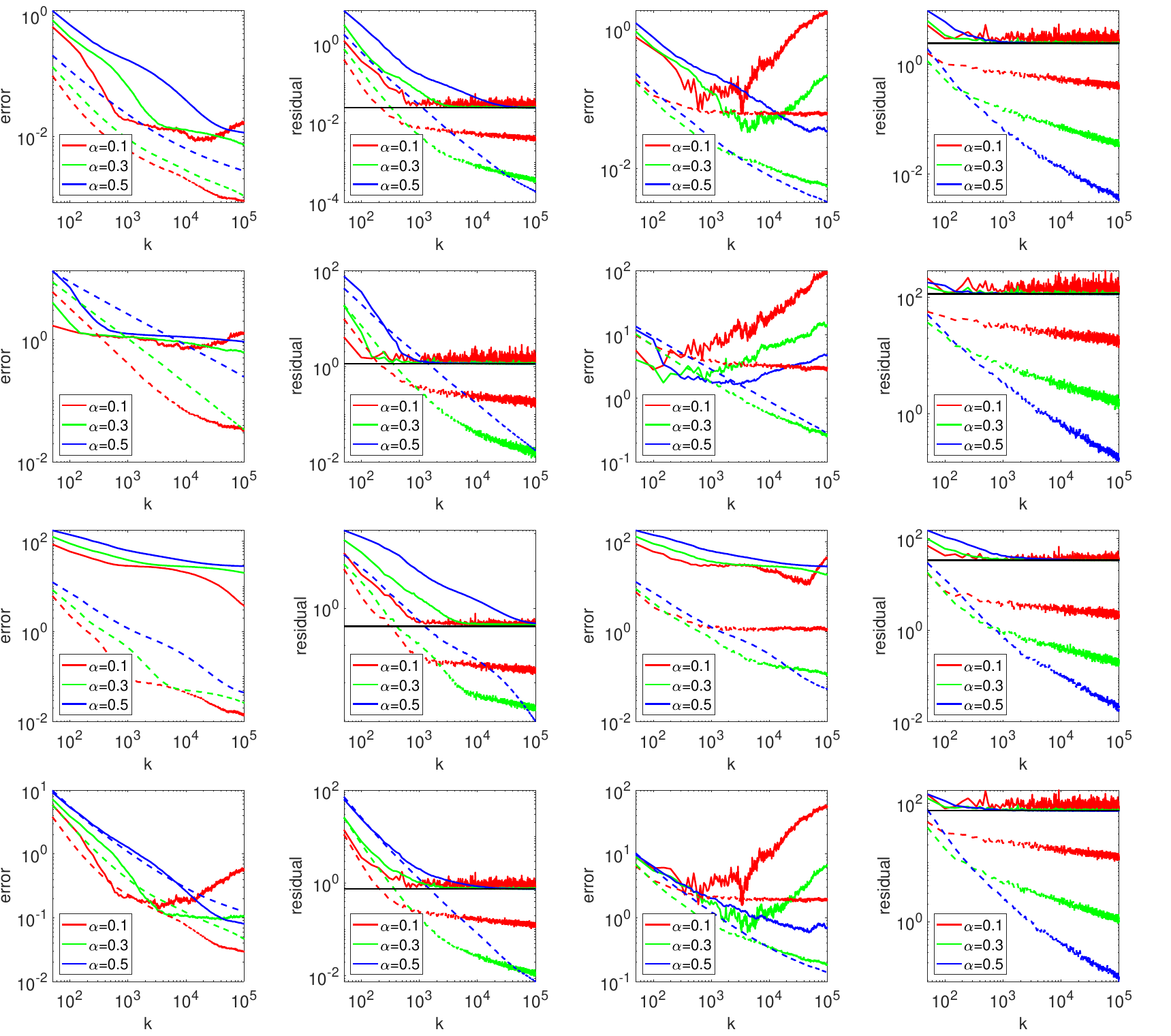}
\caption{The decay of the squared error $\|x_k^\delta-x^\dag\|^2$ and residual $\|Ax_k^\delta-y^\delta\|^2$
and their variance components $\E[\|x_k^\delta-\E[x_k^\delta]\|^2]$ and $\E[\|A(x_k^\delta-\E[x_k^\delta])\|^2]$ versus
the SGD iteration number $k$. The solid and dashed curves denote the squared quantity and the variance components,
respectively, and the black curve indicates the discrepancy $\delta^2=\|y^\delta-y^\dag\|^2$. The first two columns are
for the noise level $\delta=\mbox{5e-3}$ and the last two columns are for the noise level $\delta=\mbox{5e-2}$. The rows
from top to bottom refer to \texttt{phillips}, \texttt{gravity}, \texttt{shaw} and \texttt{smoothed-phillips}, respectively.}\label{fig:decomp0}
\end{figure}

\subsection{Independent run}

The convergence analysis in Theorem \ref{thm:dp:conv} requires a SGD iterate $x_{k(\delta)}^\delta$ independent of
the stopping index $k(\delta)$ determined by the discrepancy principle \eqref{eqn:dp}. In practice this can be achieved by an
independent run of SGD, at the expense of slightly increasing the computational effort. Now we
examine the impact of this choice, and we denote by DP and i-DP the SGD iterate used in
\eqref{eqn:dp} and that by an independent SGD run, respectively. The relevant
numerical results are presented in Tables \ref{tab:ind-phillips}--\ref{tab:ind-philsm}, where the
numbers outside and inside the bracket denote $e_{\rm sgd}$ and $\mathrm{std}(e_{\rm sgd})$, respectively.
It is observed that DP gives only slightly better results in terms of the mean, but its standard deviation
${\rm std}(e_{\rm sgd})$ is generally much smaller than that by i-DP. Nonetheless,
both the mean $\rm e_{\rm sgd}$ and the standard deviation ${\rm std}(e_{\rm sgd})$ of i-DP are
decreasing steadily as the noise level $\delta$ decreases to 0, confirming the convergence
result in Theorem \ref{thm:dp:conv}.

The difference is more clearly visualised in the boxplots in Fig. \ref{fig:ind} (for
\texttt{phillips} with two noise levels). A close look shows that the mean and percentile
are fairly close to each other, but the i-DP result tends to have far more
outliers lying above the box (marked by red cross in the plots). This is attributed to the fact that $k(\delta)$
determined by the discrepancy principle \eqref{eqn:dp} is occasionally too small for an independent SGD run, and thus
the corresponding residual is far above the target noise level in the discrepancy principle \eqref{eqn:dp};
see the boxplots in the last column of Fig. \ref{fig:ind}. That is, the outliers are due to stopping
too early. This agrees with the observation that one iteration step of SGD has only a small
effect on the high frequency components (because of the scaling with the corresponding small singular values).
Thus, small $\|Ax_k^\delta - y^\dagger\|$ for $k\ll k^*(\delta)$ implies that also $\| x_k^\delta -
x^\dagger\|$ is small. Although not presented, we note that this behavior is
observed for all the examples at different noise levels. Thus, in practice, using the SGD iterate directly
from the path for \eqref{eqn:dp} is preferred, taking into account both accuracy and
computational efficiency. It is an interesting theoretical question to analyze the convergence
(and convergence rates) of the SGD iterate by
\eqref{eqn:dp}.

\begin{table}[hbt!]
\centering
\caption{Comparison between DP and i-DP for \texttt{phillips}.\label{tab:ind-phillips}}
\setlength{\tabcolsep}{4pt}
\begin{tabular}{c|cc|cc|cc|}
\toprule
\multicolumn{1}{c}{}&\multicolumn{2}{c}{$\alpha=0.1$} & \multicolumn{2}{c}{$\alpha=0.3$} & \multicolumn{2}{c}{$\alpha=0.5$}\\
\cmidrule(l){2-3} \cmidrule(l){4-5} \cmidrule(l){6-7}
$\delta$ & DP  & i-DP & DP & i-DP & DP & i-DP\\
1e-3 &8.60e-3 (4.53e-3) &  1.12e-2 (1.18e-2) &  8.53e-3 (4.42e-3) &  1.28e-2 (1.88e-2) &  8.34e-3 (4.60e-3) &  1.28e-2 (1.55e-2)\\
5e-3 &1.70e-2 (8.41e-3) &  2.31e-2 (2.43e-2) &  2.31e-2 (8.81e-3) &  3.43e-2 (3.54e-2) &  2.48e-2 (7.38e-3) &  4.17e-2 (3.63e-2)\\
1e-2 &2.82e-2 (1.62e-2) &  4.35e-2 (4.44e-2) &  4.72e-2 (2.07e-2) &  6.43e-2 (5.67e-2) &  5.78e-2 (2.04e-2) &  6.85e-2 (5.66e-2)\\
5e-2 &1.41e-1 (9.70e-2) &  1.53e-1 (8.97e-2) &  1.49e-1 (9.01e-2) &  1.80e-1 (1.25e-1) &  2.11e-1 (9.69e-2) &  2.47e-1 (1.93e-1)\\
\bottomrule
\end{tabular}
\end{table}

\begin{table}[hbt!]
\centering
\caption{Comparison between DP and i-DP for \texttt{gravity}.\label{tab:ind-gravity}}
\setlength{\tabcolsep}{4pt}
\begin{tabular}{c|cc|cc|cc|}
\toprule
\multicolumn{1}{c}{}&\multicolumn{2}{c}{$\alpha=0.1$} & \multicolumn{2}{c}{$\alpha=0.3$} & \multicolumn{2}{c}{$\alpha=0.5$}\\
\cmidrule(l){2-3} \cmidrule(l){4-5} \cmidrule(l){6-7}
$\delta$ & DP  & i-DP & DP & i-DP & DP & i-DP\\
1e-3 &6.71e-1  (2.61e-1) &  9.30e-1  (7.45e-1)  &  7.46e-1  (2.73e-1)&  1.03e0  (8.04e-1) &  7.78e-1  (2.49e-1) & 1.00e0  (7.23e-1)\\
5e-3 &2.00e0   (8.91e-1) &  2.43e0   (1.39e0)   &  2.53e0   (1.12e0) & 3.74e0   (2.62e0)  &  2.76e0   (1.14e0)  &3.44e0   (2.36e0) \\
1e-2 &3.12e0   (1.57e0)  &  4.03e0    (2.54e0)  &  4.33e0   (1.92e0) & 5.24e0   (3.13e0)  &  4.74e0   (2.07e0)  &6.98e0   (4.17e0) \\
5e-2 &9.07e0   (5.31e0)  &  1.01e1    (5.49e0)  &  1.15e1   (6.61e0) & 1.19e1   (8.16e0)  &  1.52e1   (7.46e0)  &1.72e1   (1.10e1) \\
\bottomrule
\end{tabular}
\end{table}

\begin{table}[hbt!]
\centering
\caption{Comparison between DP and i-DP for \texttt{shaw}.\label{tab:ind-shaw}}
\setlength{\tabcolsep}{4pt}
\begin{tabular}{c|cc|cc|cc|}
\toprule
\multicolumn{1}{c}{}&\multicolumn{2}{c}{$\alpha=0.1$} & \multicolumn{2}{c}{$\alpha=0.3$} & \multicolumn{2}{c}{$\alpha=0.5$}\\
\cmidrule(l){2-3} \cmidrule(l){4-5} \cmidrule(l){6-7}
$\delta$ & DP  & i-DP & DP & i-DP & DP & i-DP\\
1e-3 & 8.29e0   (9.35e-2) &  8.30e0  (3.29e-1)  & 8.47e0   (5.59e-2) &  8.50e0  (2.67e-1)   &  2.01e1   (5.64e-1) &  2.00e1  (5.25e-1)\\
5e-3 & 2.77e1   (1.24e0)  & 2.77e1   (1.27e0)   & 2.80e1   (1.16e0)  & 2.81e1   (1.31e0)    &  2.82e1   (1.02e0)  & 2.80e1   (1.22e0) \\
1e-2 & 2.96e1   (1.65e0)  & 3.03e1   (2.58e0)   & 3.10e1   (1.14e0)  & 3.13e1   (2.74e0)    &  3.12e1   (1.08e0)  & 3.16e1   (2.44e0) \\
5e-2 & 5.02e1   (1.08e1)  & 5.34e1   (1.53e1)   & 6.07e1   (8.08e0)  & 6.19e1   (1.23e1)    &  6.70e1   (7.41e0)  & 7.04e1   (1.35e1) \\
\bottomrule
\end{tabular}
\end{table}

\begin{table}[hbt!]
\centering
\caption{Comparison between DP and i-DP for \texttt{smoothed-phillips}.\label{tab:ind-philsm}}
\setlength{\tabcolsep}{4pt}
\begin{tabular}{c|cc|cc|cc|}
\toprule
\multicolumn{1}{c}{}&\multicolumn{2}{c}{$\alpha=0.1$} & \multicolumn{2}{c}{$\alpha=0.3$} & \multicolumn{2}{c}{$\alpha=0.5$}\\
\cmidrule(l){2-3} \cmidrule(l){4-5} \cmidrule(l){6-7}
$\delta$ & DP  & i-DP & DP & i-DP & DP & i-DP\\
1e-3 & 1.63e-1   (6.87e-2) &  1.92e-1   (1.27e-1) &  1.59e-1   (5.88e-2) &  2.00e-1   (1.30e-1) &  1.55e-1   (6.09e-2) & 1.93e-1   (1.88e-1)\\
5e-3 & 3.92e-1   (2.08e-1) &  4.68e-1   (3.47e-1) &  5.06e-1   (2.05e-1) &  6.54e-1   (4.79e-1) &  4.92e-1   (1.99e-1) & 7.51e-1   (5.73e-1)\\
1e-2 & 5.95e-1   (2.64e-1) &  8.12e-1   (5.04e-1) &  8.57e-1   (3.73e-1) &  1.22e0    (1.03e0)  &  9.46e-1   (3.93e-1) & 1.46e0    (1.13e0) \\
5e-2 & 2.98e0    (1.44e0)  &  3.25e0    (1.52e0)  &  3.20e0    (1.51e0)  &  3.25e0    (1.94e0)  &  4.35e0    (2.13e0)  & 4.59e0    (3.29e0) \\
\bottomrule
\end{tabular}
\end{table}

\begin{figure}
\centering
\includegraphics[width=0.85\textwidth,trim={0.5cm 0.2cm 1cm 0cm}]{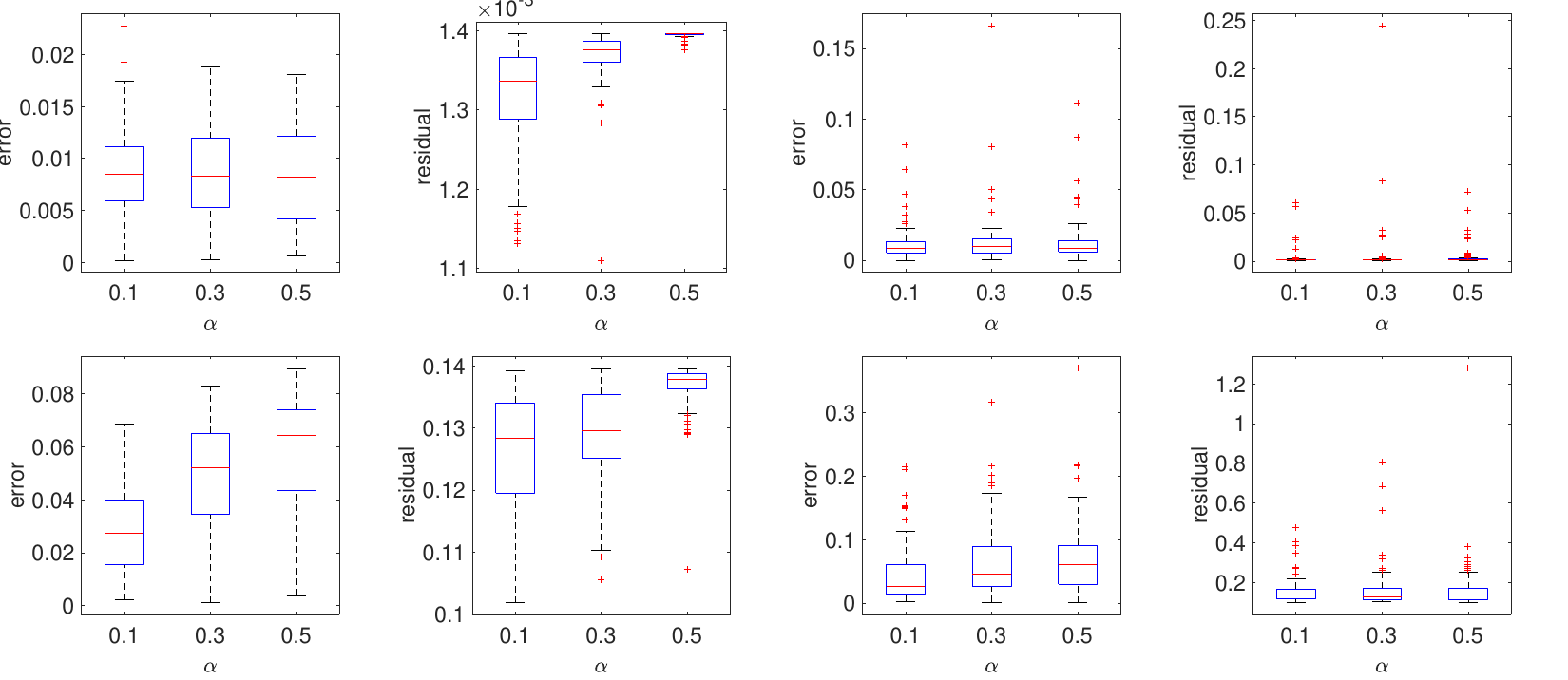}
\caption{Boxplots for the error $\|x_{k(\delta)}^\delta-x^\dag\|^2$ and the residual $\|Ax_{k(\delta)}^\delta-y^\delta\|^2$ for DP (the first two columns)
and i-DP (the last two columns), for \texttt{phillips} at two noise levels, i.e., $\delta=\mbox{1e-3}$ (top) and $\delta=\mbox{1e-2}$ (bottom).}\label{fig:ind}
\end{figure}

\section{Concluding remarks}
In this work, we have presented a preliminary study on the discrepancy principle as
an \textit{a posteriori} stopping rule for the popular stochastic gradient descent for solving
linear inverse problems. We proved a finite-iteration termination property of the principle,
and a consistency result in high probability for an independent version of discrepancy principle.
Several numerical experiments indicate the feasibility of the rule as a stopping criterion.

There are several outstanding questions that deserve further research. First, one important question is the convergence
of the dependent version of the discrepancy principle, and convergence rates (and also
optimality, if possible!). This would put the discrepancy principle on a firm mathematical basis. {Second,
it is of much interest to study stochastic gradient descent for inverse problems with random noise, with either
\textit{a priori} or \textit{a posteriori} stopping rules. In particular, in this context, the discrepancy
principle may have to be properly adapted; see the works \cite{BlanchardMathe:2012,HarrachJahn:2020}
for interesting discussions with deterministic inversion techniques. Third, the analysis so far does not
cover the critical case $\alpha=1$ in the stepsize schedule. This choice is often adopted in the context
of stochastic approximation \cite{KushnerYin:2003} for optimal asymptotic behaviour, but it is unclear
whether the discrepancy principle can be applied then.}

\bibliographystyle{abbrv}
\bibliography{references}
\end{document}